\documentclass[10pt]{amsart}

\usepackage[dvips]{graphicx}
\usepackage{amssymb,color,latexsym}
\usepackage[centertags]{amsmath}
\usepackage{amsfonts}
\usepackage[spanish,english]{babel}  
\usepackage[T1]{fontenc}

\newtheorem{thm}{Theorem}

\newtheorem{prop}{Proposition}

\newcommand{\ma}{\mathcal{M}}
\newcommand{\fl}{\mathcal{F}}

\newcommand{\gr}{\mathcal{G_M}}
\newcommand{\Mon}{\mathrm{Mon}}
\newcommand{\oo}{\mathcal{O}}
\newcommand{\Orb}{\mathrm{Orb}}
\newcommand{\Tr}{\mathrm{Tr}}
\newcommand{\Le}{\mathrm{Le}}

\newcommand{\A}{\mathcal{A}}

\begin{document}



\title{Truncation symmetry type graphs}

\author{Mar\'ia del R\'io Francos} 

\address{\small Institute of Mathematics Physics and Mechanics, 
University of Ljubljana, Slovenia, Jadranska 19, Ljubljana 1000, Slovenia\\
	    {\tt maria.delrio@fmf.uni-lj.si}
}

\date{}


\maketitle

\begin{abstract}
There are operations that transform a map $\ma$ (an embedding of a graph on a surface) into another map in the same surface, modifying its structure and consequently its set of flags $\fl(\ma)$. 
For instance, by truncating all the vertices of a map $\ma$, each flag in $\fl(\ma)$ is divided into three flags of the truncated map. Orbani\'c, Pellicer and Weiss studied the truncation of $k$-orbit maps for $k \leq 3$. 
They introduced the notion of $T$-compatible maps in order to give a necessary condition for a truncation of a $k$-orbit map to be either $k$-, $\frac{3k}{2}$- or $3k$-orbit map.
Using a similar notion, by introducing an appropriate partition on the set of flags of the maps, we extend the results on truncation of $k$-orbit maps for $k \leq 7$ and $k=9$.\\

\end{abstract}

\section{Introduction.}

From the topological point of view, we define a map by a cellular embedding of a connected graph on a surface. The most symmetric maps are those for which its group of automorphisms acts transitively on its elements, called flags. This type of maps are well known under the name of regular (or reflexible) maps. 
Highly symmetric maps, as the regular and chiral (2-orbit maps with maximum degree of symmetry by rotation) ones, have been extensively studied, \cite{McMullen,arp,Daniel-chiral}. 
Little is known about maps that are neither regular nor chiral. 
Duarte and Hubard studied in \cite{rui} and \cite{2Polyh}, respectively, all seven types of 2-orbit maps, in different contexts. 

Combinatorially, a map is completely described by a cubic edge-coloured graph, as Vince describes it in \cite{CombMaps}. Such graph is known as the {\em flag graph} of the map. In 1982, Lins proposed an equivalent description to a flag graph in \cite{Lins}. Other interesting characteristics regarding the flag graph are described in \cite{ColorfulPolytopes}.

Furthermore, each map has associated a graph obtained as the quotient of its flag graph, under the action of the group of automorphisms of the map, determining what we call the {\em  symmetry type graph} of the map. In \cite{MedSymTypeGph} are described some properties of these symmetry type graphs. A strategy of how to generate symmetry type graphs is shown in \cite{CompSymTypeGraph}. Dress and Huson refer to such graphs as the Delaney-Dress symbol, \cite{DressHuson87}. Dress and Brinkmann give an application to mathematical chemistry in \cite{DrBr96}.
Orbani\'c, Pellicer, Pisanski and Tucker (2011), show the 14 symmetry type graphs of edge-transitive maps, \cite{Edge-trans}.
Edge-transitive maps were studied in \cite{edge-trans} by Siran, Tucker and Watkins. 
Such maps have either 1, 2 or 4 orbits of flags under the action of the automorphism group.

A map can be transformed into another map by applying an operation on it. The truncation by the vertices of the map is an example of such operation. Even though a map and its truncated map are different to each other, the truncated map inherits the symmetries from the original map. If the vertices of a map $\ma$ are truncated up to the midpoint of its edges, then the resulting map is called the medial map of $\ma$. Moreover, when the truncation goes further than the midpoint on the edges of the map $\ma$, then the obtained map is known as the leapfrog map of $\ma$, \cite{Leapfrog}. Furthermore, if the vertices of $\ma$ become the faces of the truncated map, then the resulting map is the dual map of $\ma$. 
For instance, by truncating regular polyhedra, we can obtain some of the 13 Archimedean solids as is discussed in \cite{RegPoly} and \cite{ArchimSolids}. 

The truncation map $\Tr(\ma)$, of a map $\ma$, can be described by a division of the fundamental triangles on the surface that represent the flags of $\ma$, \cite{k-orbitM}. 
Consequently, we can define a transformation from the flag graph of $\ma$ to the flag graph of its truncation map. 
According to such transformation of the flags of $\fl(\ma)$ to define the set of flags $\fl(\Tr(\ma))$ of $\Tr(\ma)$, we follow a local arrangement of the flags in the truncated map, in addition to certain necessary conditions on $\ma$ and $\Tr(\ma)$, 
Orbani\'c, Pellicer and Weiss showed the results on truncation of regular, 2-orbit and 3-orbit maps, stated as Theorem \ref{orbitsTr(M)} and Propositions \ref{2-orbitTr}, \ref{3-orbitTr_2-orbitM} and \ref{3-orbitTr} in this paper. 
The goal of this paper is to extend these results on truncation of $k$-orbit maps for $k \leq 7$ and $k=9$, due to space, we leave on aside the truncation of 8-orbit maps. For this extension we will use the same local arrangement of flags used in \cite{k-orbitM}.

The paper is divided into five sections. In Section \ref{sec:Maps}, a map and its flag graph are defined along with their properties. The monodromy group of a map and its action on the set of flags are defined as in \cite{MonGp_Self-Inv}. In Section \ref{sec:STG} the symmetry type graph of a map is defined as a quotient of its flag graph, \cite{MedSymTypeGph}. In Section \ref{sec:TruncLeap} are defined operations on maps such as the dual, truncation, and the composition of these two. The Section \ref{subsec:Trunc}, related to the truncation operation, is divided into 5 subsections where are analized the results in \cite{k-orbitM} (on truncation of $k$-orbit maps) in addition with the extension of these results, given in a series of Propositions (\ref{4-orbTr}--\ref{9-orbTr(9)}). Then, we obtain a classification of the possible symmetry type graphs for the truncation of $k$-orbit maps, with $k\leq 7$ and $k=9$ (Table \ref{Tr(M)}). Later, in Section \ref{subsec:Lp}, are defined the two-dimensional subdivision of a map and the leapfrog map, obtaining a classification of the possible symmetry type graphs for the leapfrog of $k$-orbit maps, with $k\leq 7$ and $k=9$ (Table \ref{Le(M)}).

\section{Maps.}
\label{sec:Maps}

A {\em map} $\ma$ is a 2-cell embedding of a connected graph on a surface. The vertices and edges of the map $\ma$ are those from the embedded graph, and the faces of $\ma$ are homeomorphic to an open disc, such that each edge of a face is incident with at most two faces of $\ma$. 

Consider the barycentric subdivision $\mathcal{B}$ of the faces of $\ma$. Let $\Phi$ be a triangle in the subdivision $\mathcal{B}$, we shall label the vertices of $\Phi$ by $\Phi_0$, $\Phi_1$ and $\Phi_2$ if these correspond to a vertex, edge and face mutually incident in $\ma$. 
Each triangle $\Phi$ is adjacent to exactly three other that differ on either a vertex, an edge or a face of the map.
There are exactly four triangles around each edge of $\ma$.  
Define a {\em flag} $\Phi$ of the map $\ma$ as an ordered triple $(\Phi_0,\Phi_1,\Phi_2)$ of a vertex, edge and face mutually incident  in $\ma$, and denote as $\fl(\ma)$ the set of all flags of $\ma$.
Hence, every flag $\Phi\in\fl(\ma)$ is adjacent to three other flags, called {\em $i$-adjacent} flags of $\Phi$, which corresponding triples differ on the entry $\Phi_i$ from $\Phi$, and  are denoted by $\Phi^i$, with $i=0,1,2$. In other words, the flags of $\ma$ are related to the faces of $\mathcal{B}$. The set of all flags of $\ma$ is denoted by $\fl(\ma)$.

To each map $\ma$ corresponds a subgroup of the permutation group $Sym(\fl(\ma))$, generated by three permutations $s_0, s_1, s_2$ such that for each flag $\Phi \in \fl(\ma)$ the generator $s_i$ acts on the right as follows
                             $$\Phi^{s_i} = \Phi \cdot s_i = \Phi^i,$$
with $i=0,1,2$. 
Note that $(\Phi^i)^i=\Phi$, and that $\Phi^{0,2}=\Phi^{2,0}$ (since each edge in $\ma$ has exactly four triangles around it). Then, it follows that the generators $s_0, s_1, s_2$ are fixed-point free involutions in $Sym(\fl(\ma))$, such that $s_0s_2=s_2s_0$ is also a fixed-point free involution of the flags of $\ma$. 
Moreover, the involutions $s_0, s_1, s_2$ generate a subgroup, $\Mon(\ma)$, of $Sym(\fl(\ma))$ known as the {\em monodromy (or connection)} group of the map $\ma$, \cite{MonGp_Self-Inv}. Since maps are connected, the action of $\Mon(\ma)$ on $\fl(\ma)$ is transitive.

Given a map $\ma$ we can construct a graph $\gr$ in the following way. 
The vertex set of $\gr$ is $\fl(\ma)$ and its edges are labelled by the subscripts of the generators of $\Mon(\ma)$ in a natural way. In other words, two vertices $\Phi, \Psi \in \fl(\ma)$ are adjacent by an edge of colour $i=0,1,2$ in $\gr$, if and only if $\Phi^{s_i} = \Psi$ in $\ma$. The graph $\gr$ is known as the flag graph of the map $\ma$. 
Moreover, if $\Phi, \Psi \in \fl(\ma)$ are two vertices of $\gr$, then a walk along the edges $i_0, i_1, \dots, i_t$ that starts at $\Phi$ and ends at $\Psi$ is a word $w = s_{i_0}s_{i_1} \cdots s_{i_t} \in \Mon(\ma)$, defining inductively the action of $\Mon(\ma)$ on the set $\fl(\ma)$ as follows
$$\Phi^w =  (\Phi^{i_0})\cdot{s_{i_1} \cdots {s_{i_t}}} = \Phi^{i_0. i_1, \dots, i_t} =\Psi.$$

Observe that in $\gr$ the edges of a given colour form a perfect matching (an independent set of edges containing all the vertices of the graph), and the union of two sets of edges of different colour is a subgraph whose components are even cycles; such subgraph is known as a {\em 2-factor} of $\gr$. In particular, note that since $(s_0s_2)^2 = 1$ and $s_0s_2$ is fixed-point free, the cycles with edges of alternating colours 0 and 2 are all of length four and these 4-cycles define the set of edges of $\ma$. 
In other words, the edges of $\ma$ can be identified with the orbits of $\fl(\ma)$ under the action of the subgroup generated by the involutions $s_0$ and $s_2$; that is, $E(\ma)= \{ \Phi^{\langle s_0,s_2 \rangle} \mid \Phi \in \fl(\ma) \}$.

Similarly, we find that the vertices and faces of $\ma$ are identified with the respective orbits of the subgroups $\langle s_1,s_2 \rangle$ and $\langle s_0,s_1 \rangle$ on $\fl(\ma)$. That is, $V(\ma)= \{ \Phi^{\langle s_1,s_2 \rangle}  \mid \Phi \in \fl(\ma)  \}$ and $F(\ma)= \{ \Phi^{\langle s_0,s_1 \rangle}  \mid \Phi \in \fl(\ma) \}$.
Thus, the group $\langle s_0, s_1, s_2 \rangle$ induces a transitive action on the set of all vertices, edges and faces of $\ma$.

Let $\Gamma(\ma)$ denote the group of automorphisms of the map $\ma$. An {\em automorphism} $\gamma$ of $\ma$ is a bijection between the vertices, edges and faces of $\ma$ that preserves the map. In fact, every $\gamma\in\Gamma(\ma)$ ``commutes'' with the distinguished generators $s_0, s_1, s_2$ of $\Mon(\ma)$, i.e. $\Phi^{s_i}\gamma=(\Phi\gamma)^{s_i}$, \cite{MonGp_Self-Inv}. It is easy to see that the only automorphism that fixes any flag is the identity one. Therefore, the action of $\Gamma(\ma)$ on $\fl(\ma)$ is free (semi-regular) and partitions $\fl(\ma)$ into orbits of the same size. If there are $k$ orbits of flags under the action of $\Gamma(\ma)$, then $\ma$ is said to be a {\em $k$-orbit map}. Moreover, if $\Gamma(\ma)$ acts transitively on $\fl(\ma)$, then $\ma$ is called a {\em regular} map.

As the automorphism group $\Gamma(\ma)$ of $\ma$ acts on the vertices, edges and faces of $\ma$ preserving their incidence, every automorphism of $\ma$ induces a bijection between the flags of $\ma$ preserving their adjacencies. Consequently $\Gamma(\ma)$ is isomorphic to the edge-colour preserving automorphism group of $\gr$, which consists of all the automorphisms of $\gr$ that send two vertices adjacent by an edge of colour $i$ into other two vertices adjacent by an edge of the same colour $i$, for each $i=0,1,2$.

\section{Symmetry type graph.}
\label{sec:STG}

Recall that the automorphism group $\Gamma(\ma)$ of the map $\ma$ partitions the set of flags $\fl(\ma)$ of $\ma$ into orbits of the same size. 
Let $\Orb(\ma) = \{\oo_{\Phi}|\Phi\in\fl(\ma)\}$ denote the set of all the orbits of $\fl(\ma)$ under the action of $\Gamma(\ma)$. 

From the flag graph $\gr$ we obtain a quotient graph $T(\ma)$ given by the action of $\Gamma(\ma)$ on $\gr$ in the following way. The vertices of $T(\ma)$ are the orbits of $\fl(\ma)$ under $\Gamma(\ma)$ and two orbits $\oo_{\Phi}, \oo_{\Psi} \in \Orb(\ma)$ are adjacent by an edge of colour $i=0,1,2$ if and only if there are flags $\Phi' \in \oo_{\Phi}$ and $\Psi' \in \oo_{\Psi}$ such that they are $i$-adjacent in $\gr$; note that this adjacency is well-defined.
The graph $T(\ma)$ might have parallel edges (edges with the same end-points) or also semi-edges (edges with a single end-point), but $T(\ma)$ has no loops. We refer to the graph $T(\ma)$ as the {\em symmetry type graph} of $\ma$, and we also say that $\ma$ is a $k$-orbit map of type $T$. For instance, the symmetry type graph of a regular map is a graph with a single vertex and three semi-edges, one of each colour 0, 1 and 2. If $\ma$ is a $k$-orbit map, then $T(\ma)$ is a cubic graph with exactly $k$ vertices and the edges are properly coloured with three colours. 

For every $w\in\Mon(\ma)$ and $\Phi\in\fl(\ma)$ the action of $\Mon(\ma)$ on the set $\Orb(\ma)$ is defined as $\oo_{\Phi}\cdot w=\oo_{\Phi^w}$. This action is transitive, as is the action of $\Mon(\ma)$ on $\fl(\ma)$. Since $\gr$ is a connected graph, then its corresponding symmetry type graph $T(\ma)$ is connected as well. 
 The number of types of $k$-orbit maps is the number of connected cubic graphs with $k$ vertices, properly three edge-coloured, where the colours 0 and 2 are as in the Figure \ref{2-factors(0,2)}.
\begin{figure}[htbp]
\begin{center}
\vspace{0.5cm}
\includegraphics[width=11cm]{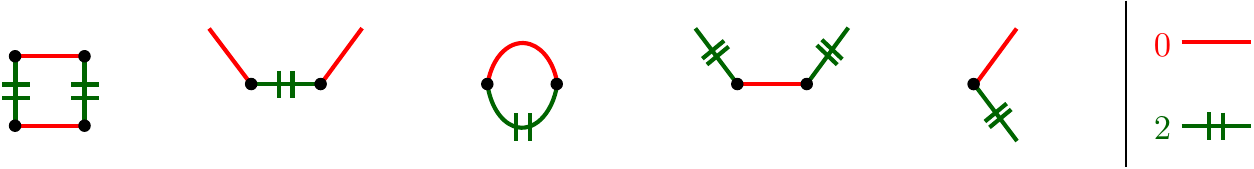}
\vspace{0.5cm}
\caption{Possible quotients of 0-2 coloured 4-cycles of $\gr$.}
\label{2-factors(0,2)}
\end{center}
\end{figure}

 In \cite{MedSymTypeGph} are shown all symmetry type graphs for $k\leq5$, along with more properties of these graphs, such as the following Theorem.

\begin{thm}
\label{face-orbT(M)}
Let $\ma$ be a map with symmetry type graph $T(\ma)$. Then, the number of connected components in the 2-factor of $T(\ma)$ of colours $i$ and $j$, with $i,j \in \{0,1,2\}$ and $i \neq j$, is the same as the number of orbits of the $l$-faces of $\ma$, where $l \in \{0,1,2\}$ and $l \neq i,j$. 
\end{thm}

The notation for the symmetry type graphs with $k$ vertices is followed from \cite{MedSymTypeGph} as well.


\section{Operations on maps.}
\label{sec:TruncLeap}

Given a map $\ma$ it is possible to find other maps in the same surface as $\ma$, by applying an operation to $\ma$. One of the most recalled operations on maps is the dual operation. Two maps $\ma$ and $\mathcal{N}$ are said to be dual of each other if there exists a bijection $\delta$ between the set of flags $\fl(\ma)$ of $\ma$ and the set of flags $\fl(\mathcal{N})$ of $\mathcal{N}$ such that for each $\Phi\in\fl(\ma)$ and each $i\in\{0,1,2\}$, $\Phi^i\delta=(\Phi\delta)^{2-i}$. 
The {\em dual map} of $\ma$ is denoted by $\ma^*$. Note that $(\ma^*)^* \cong \ma$. 

If there exist a bijection $\delta$ from a map $\ma$ to itself, then $\ma \cong \ma^*$ and it is called a {\em self-dual map}. In this case we name $\delta$ a {\em duality} of $\ma$. Observe that for a self-dual map $\ma$, if $\delta_1$ and $\delta_2$ are two dualities of $\ma$, then the product of them is an automorphism of $\ma$. In particular, the square of any duality is an automorphism, and hence the order of any duality is even.

In terms of the flag graphs, the bijection $\delta$ can be regarded as a bijection between the vertices of $\gr$ and the vertices of $\mathcal{G_{\ma^*}}$ that sends edges of colour $i$ of $\gr$ to edges of colour $2-i$ of $\mathcal{G_{\ma^*}}$, for each $i \in \{0,1,2\}$. 

In \cite{MedSymTypeGph} was considered the problem of classification of $k$-orbit medial maps. Here we address the same problem for truncation and leapfrog maps, extending the results in \cite{k-orbitM} on truncation of $k$-orbit maps.

\subsection{Truncation.}
\label{subsec:Trunc}

Geometrically, when we apply truncation to a map $\ma$, the vertices of $\ma$ are replaced by faces, and the faces of $\ma$ become faces with twice the number of vertices than the original ones. An example is the truncated octahedron, shown in Figure \ref{OctTrunc}, it has six square faces and eight hexagonal faces that correspond to the vertices and faces of the octahedron, respectively. Hence, there is a correspondence between the set of faces of the truncated map $\Tr(\ma)$ of $\ma$ with the set of vertices and faces of $\ma$. That is, $F(\Tr(\ma)) = V(\ma)\cup F(\ma)$. 

\begin{figure}[htbp]
\begin{center}
\includegraphics[width=9cm]{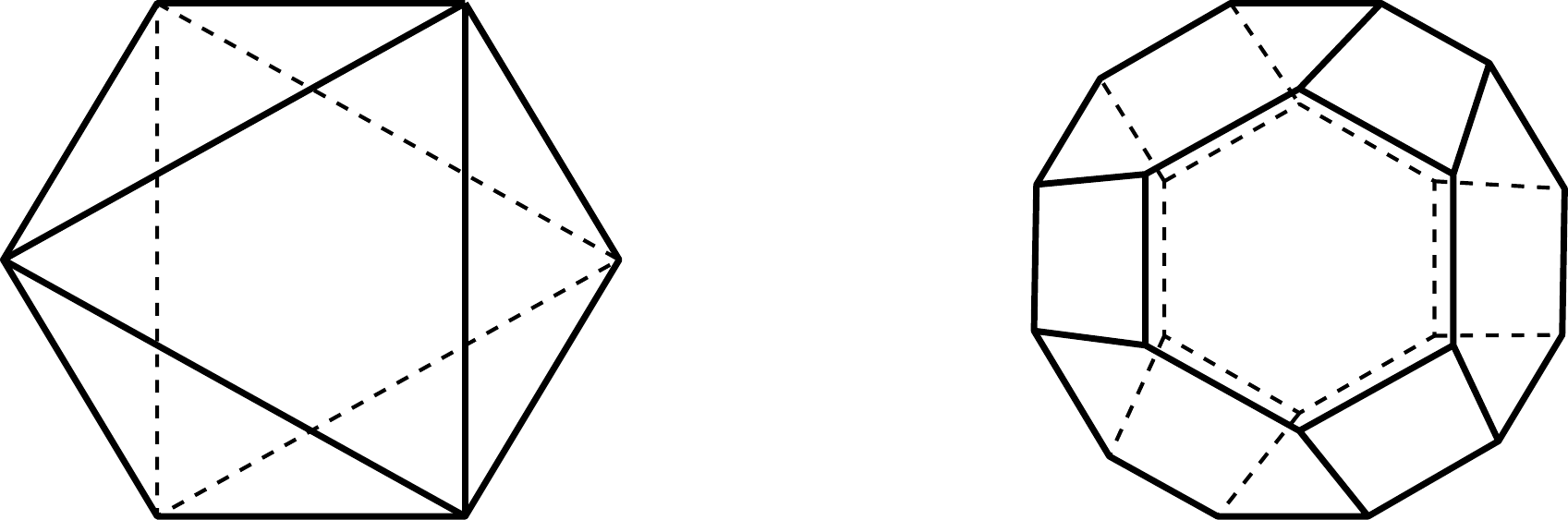}
\caption{Octahedron and truncated octahedron.}
\label{OctTrunc}
\end{center}
\end{figure}
 
Furthermore, for each edge of $\ma$ there are exactly two vertices of $\Tr(\ma)$, two of these vertices have an edge joining them if either both belong to a common edge of $\ma$ or if the corresponding edges in $\ma$  share a vertex and belong to the same face. 
Then, the sets of edges and vertices of $\Tr(\ma)$ are 
 $$E(\Tr(\ma))=E(\ma)\cup\{\{\Phi_0,\Phi_2\}|\Phi\in\fl(\ma)\} \; \text{ and }$$ 
$$V(\Tr(\ma)) =\{\{\Phi_0,\Phi_1\}|\Phi\in\fl(\ma)\},$$
respectively.
Each vertex of $\Tr(\ma)$ has valency 3 and therefore, the truncated map $\Tr(\ma)$ contains $2|E(\ma)|$ vertices and $3|E(\ma)|$ edges.

In the Figure \ref{TruncFlags} is depicted how on the surface every fundamental triangle for $\ma$ is divided into three smaller fundamental triangles for the truncated map $\Tr(\ma)$. 
\begin{figure}[htbp]
\begin{center}
\includegraphics[width=5cm]{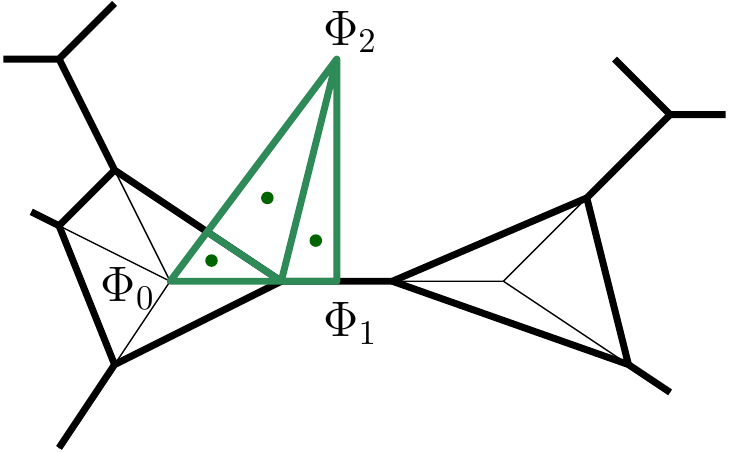}
\caption{The three flags of $\fl(\Tr(\ma))$ corresponding to the flag $\Phi =(\Phi_0,\Phi_1,\Phi_2) \in \fl(\ma)$.}
\label{TruncFlags}
\end{center}
\end{figure}
This is, for each flag $\Phi=(\Phi_0,\Phi_1,\Phi_2)$ in $\fl(\ma)$, 
there are three flags $(\Phi,0):=(\{\Phi_0,\Phi_1\}, \{\Phi_0,\Phi_2\}, \Phi_0)$, $(\Phi,1):=(\{\Phi_0,\Phi_1\}, \Phi_1, \Phi_2)$ and $(\Phi,2):=(\{\Phi_0,\Phi_1\},\{\Phi_0,\Phi_2\}, \Phi_2)$ corresponding to $\Phi$ in $\fl(\Tr(\ma))$.
 The adjacencies between the flags in $\fl(\Tr(\ma))$ are given as follows. 
\begin{align*}
(\Phi,0)\cdot r_0&=(\Phi^{s_1},0), & (\Phi,0)\cdot r_1&=(\Phi^{s_2},0), & (\Phi,0)\cdot r_2&=(\Phi,2); \\
(\Phi,1)\cdot r_0&=(\Phi^{s_0},1), & (\Phi,1)\cdot r_1&=(\Phi,2),             & (\Phi,1)\cdot r_2&=(\Phi^{s_2},1);\\
(\Phi,2)\cdot r_0&=(\Phi^{s_1},2), & (\Phi,2)\cdot r_1&=(\Phi,1),             & (\Phi,2)\cdot r_2&=(\Phi,0).
\end{align*}
Notice that $r_0$, $r_1$ and $r_2$ depend of the adjacency between the flags in $\fl(\ma)$. 
Moreover, $r_0$, $r_1$ and $r_2$ are fixed-point free involutions that generate the monodromy group $\Mon(Tr(\ma))$ of $\Tr(\ma)$, where $r_0r_2=r_2r_0$ and, $(r_1r_2)^3=id$. 

Consequently, in Figure \ref{TruncAlg} is presented an algorithm to construct the flag graph of $\Tr(\ma)$ from $\gr$.
\begin{figure}[htbp]
\begin{center}
\includegraphics[width=10.5cm]{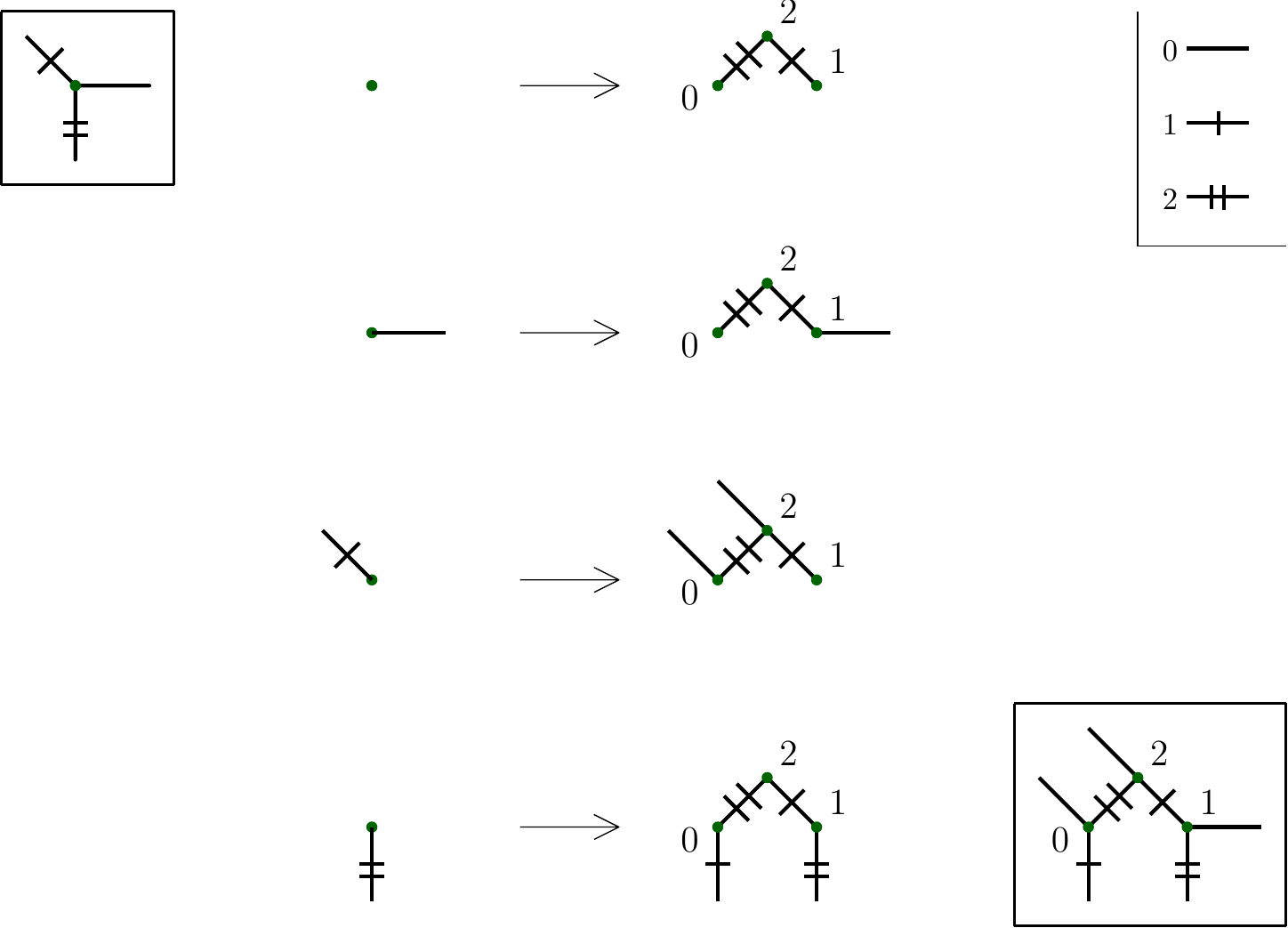}
\caption{Local representation of a flag in $\gr$, in the left. The image under the truncation operation, locally obtained, in the right.}
\label{TruncAlg}
\end{center}
\end{figure}
Such algorithm induces a partition $(\A_0, \A_2, \A_1)$ on the vertices of $\mathcal{G}_{\Tr(\ma)}$, where 
$\A_i:=\{(\Phi,i)|\Phi\in\fl(\ma)\}$ with $i=0,1,2$, and consequently we have Proposition \ref{3-0-comp}.

\begin{prop}\label{3-0-comp}
The flag graph $\mathcal{G}_{\Tr(\ma)}$, of the truncation map $\Tr(\ma)$ of any map $\ma$, can be quotient into a graph as the symmetry type graph $3^0$.
\end{prop}

\begin{proof}
Let $\A_i=\{(\Phi,i)|\Phi\in\fl(\ma)\}$ be the subset of $\fl(\Tr(\ma))$ containing all flags of $\Tr(\ma)$ of the form $(\Phi,i)$ with $i=0,1,2$. 
Then, $\fl(\Tr(\ma))=\A_0\cup\A_1\cup\A_2$ and $\A_0\cap\A_1\cup\A_2=\emptyset$. 
Hence, $(\A_0,\A_2,\A_1)$ is a partition of the set of flags $\fl(\Tr(\ma))$. 
Based on Figure \ref{TruncAlg}, it is straightforward to see that the quotient of $\mathcal{G}_{\Tr(\ma)}$ over such partition, is isomorphic to the symmetry type graph of a map with symmetry type $3^0$ (see Figure \ref{STG-3-0}).
\begin{figure}[htbp]
\begin{center}
\includegraphics[width=3.5cm]{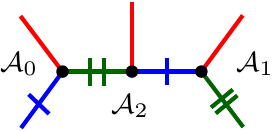}
\caption{Symmetry type graph $3^0$.}
\label{STG-3-0}
\end{center}
\end{figure}
\end{proof}

As the vertices of the map $\Tr(\ma)$ have valency three, there are exactly six flags of $\fl(\Tr(\ma))$ around each vertex in $\Tr(\ma)$. 
Recall that the vertices of the map $\Tr(\ma)$ are identified with the orbits of the subgroup $\langle r_1,r_2 \rangle$ on $\fl(\Tr(\ma))$.
Since the automorphism group of a map partitions its set of flags into orbits of the same size, then 
Theorem \ref{face-orbT(M)} implies that the 2-factors coloured by 1 and 2 in the symmetry type graph $T(\Tr(\ma))$ must be as those in Figure \ref{2-factors(1,2)Tr(M)}.

\begin{figure}[htbp]
\begin{center}
\vspace{0.5cm}
\includegraphics[width=11cm]{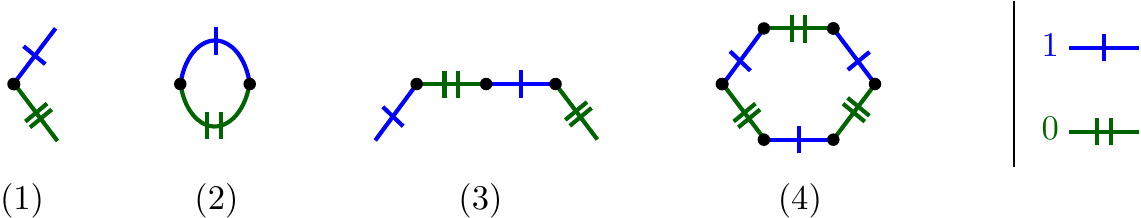}
\vspace{0.5cm}
\caption{Possible quotients of 1-2 coloured 6-cycles of $\mathcal{G}_{\Tr(\ma)}$.}
\label{2-factors(1,2)Tr(M)}
\end{center}
\end{figure}

Orbani\'c, Pellicer and Weiss proved the following Theorem and present examples of when any of this cases is possible, \cite{k-orbitM}.
\begin{thm}\cite{k-orbitM}
\label{orbitsTr(M)}
If $\ma$ is a $k$-orbit map, then the map $\Tr(\ma)$ is either a $k$-orbit, a $\frac{3k}{2}$-orbit (with $k$ even) or a $3k$-orbit map.
\end{thm}

With regards to this result, using coset enumeration, the authors in \cite{k-orbitM} show that in particular if a $k$-orbit map $\ma$ which truncation map $\Tr(\ma)$ is either a $\frac{3k}{2}$-orbit or a $k$-orbit map, then we must necessarily find a bi-partition on the vertices of $\gr$ in such way that $\gr$ can be quotient into a graph isomorphic to the symmetry type graph $2_{01}$; i.e. all vertices of $\ma$ must have even valence, and the flags of each two adjacent faces in $\ma$ can be put in different parts.

By Proposition \ref{3-0-comp}, we have that there exists a partition $(\A_0,\A_2,\A_1)$ of the vertices of the flag graph $\mathcal{G}_{\Tr(\ma)}$ of the truncation map $\Tr(\ma)$ in such way that for each vertex $\Psi\in\fl(\Tr(\ma))$ in the partition $\A_0$, the vertices $\Psi^2$ and $\Psi^{2,1}$ correspond to the partitions $\A_2$ and $\A_1$, respectively. This latter implies a {\em local arrangement in the flags of $\Tr(\ma)$: assembling the flags $\Psi,\Psi^2,\Psi^{2,1}\in\fl(\Tr(\ma))$ into a new flag $\Phi_{\Psi}\in\fl(\ma)$ where the face $\Psi_2$ will be considered as an element of $V(\ma)$, \cite{k-orbitM}.}

Finally the authors, in \cite{k-orbitM}, conclude that the flag graph $\mathcal{G}_{\Tr(\ma)}$ can be quotient in to a graph isomorphic to the symmetry type graph $3^0$ under the action of a subgroup $H\leq\Gamma(\Tr(\ma))$, of the automorphism group of $\Tr(\ma)$, if and only if $\ma$ is a regular map. 

In what follows, supported by Proposition \ref{3-0-comp} and the local arrangement on the flags of $\Tr(\ma)$ described above, 
we discuss the results of Orbani\'c, Pellicer and Weiss on truncation of regular, 2-orbit and 3-orbits maps, \cite{k-orbitM}.
Further on, we show results on the truncation of $k$-orbit maps, with $k =4,\dots,7$ and $k=9$, using the same local arrangement on the flags of $\Tr(\ma)$ that Orbani\'c, Pellicer and Weiss used for smaller $k$. 
Due to space, we leave on aside the truncation of 8-orbit maps.

\subsubsection{Truncation of regular, 2-orbit and 3-orbit maps.}

As it was said previously, 
if $\ma$ is a regular map, then $\Tr(\ma)$ is either regular or a 3-orbit map, which has symmetry type graph $3^0$ (see Figure \ref{STG-3-0}).
On another hand, if $\ma$ is a 2-orbit or a 3-orbit map, Orbani\'c, Pellicer and Weiss conclude with the following results on the truncation map of $\ma$, \cite{k-orbitM} .

\begin{prop}\cite{k-orbitM}
\label{2-orbitTr}
If the truncation $\Tr(\ma)$ of a 2-orbit map $\ma$ is a 2-orbit map, then one of the following holds
\begin{itemize}
\item[(i)] $\ma$ and $\Tr(\ma)$ are of type 2,
\item[(ii)] $\ma$ is of type $2_{01}$ and $\Tr(\ma)$ is of type $2_0$, or
\item[(iii)] $\ma$ is of type $2_2$ and $\Tr(\ma)$ is of type $2_{12}$.
\end{itemize}
\end{prop}

The symmetry type graphs 2, $2_{01}$, $2_0$, $2_2$ and $2_{12}$ are depicted in Figure \ref{OPW}.

\begin{prop}\cite{k-orbitM}
\label{3-orbitTr_2-orbitM}
If the truncation $\Tr(\ma)$ of a 2-orbit map is a 3-orbit map, then $\ma$ is of type $2_{01}$ and $\Tr(\ma)$ is of type $3^0$ (see Figure \ref{OPW}).
\end{prop}

\begin{prop}\cite{k-orbitM}
\label{3-orbitTr}
If the truncation $\Tr(\ma)$ of a 3-orbit map is a 3-orbit map, then $\ma$ and $\Tr(\ma)$ are of type $3^{02}$ (see Figure \ref{OPW}).
\end{prop}

In other words, if $\ma$ is a regular, a 2-orbit or a 3-orbit map, and $\Tr(\ma)$ has either 1, 2 or 3 flag-orbits, then the symmetry type graph of $\Tr(\ma)$ is one of the six graphs in Figure \ref{OPW}. Which in fact, are the only possible symmetry type graphs with 1, 2 and 3 vertices, consistent with the (1,2) 2-factors in Figure \ref{2-factors(1,2)Tr(M)}.
\begin{figure}[htbp]
\begin{center}
\vspace{0.5cm}
\includegraphics[width=10cm]{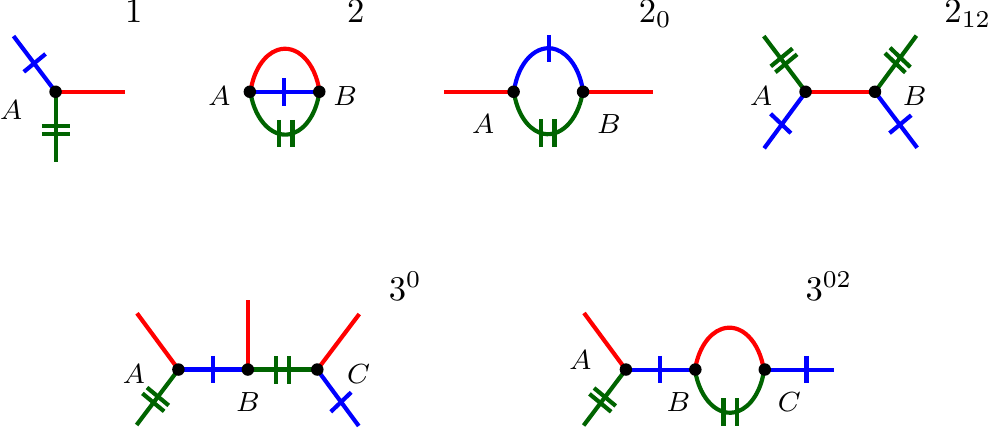}
\vspace{0.5cm}
\caption{Symmetry type graphs of $\Tr(\ma)$, for a regular map or $\ma$ as in Propositions \ref{2-orbitTr}, \ref{3-orbitTr_2-orbitM}, \ref{3-orbitTr}.}
\label{OPW}
\end{center}
\end{figure}

Orbani\'c, Pellicer and Weiss prove the Propositions \ref{2-orbitTr}, \ref{3-orbitTr_2-orbitM} and \ref{3-orbitTr} based on the following facts.
\begin{itemize}
\item[(i)] All the vertices of the map $\Tr(\ma)$ have valency 3.
\item[(ii)] The vertex set of $\ma$ is a proper subset of the set of faces of $\Tr(\ma)$ (recall that $F(\Tr(\ma)) = V(\ma)\cup F(\ma)$).
\item[(iii)] Each flag $\Phi_{\Psi}\in\fl(\ma)$ of $\ma$ is divided in exactly three flags $(\Phi_{\Psi},0)=:\Psi$, $(\Phi_{\Psi},2)=:\Psi^2$, $(\Phi_{\Psi},1)=:\Psi^{2,1} \in \fl(\Tr\ma)$ of $\Tr(\ma)$ (Figures \ref{TruncFlags} and \ref{TruncAlg}), where $$\Phi_{\Psi}:=\{\Psi,\Psi^2,\Psi^{2,1}\}.$$
\end{itemize}

This is, given that we can define a proper partition $(\A_0,\A_2,\A_1)$ of the set $\fl(\Tr(\ma))$ in such way that for each flag $\Psi\in\A_0$, the flags $\Psi^2\in\A_2$ and $\Psi^{2,1}\in\A_1$, 
consequently. 
Assume that the flag $\Psi\in\A_0$ represents a particular orbit of the flags of $\Tr(\ma)$, follow its adjacencies determined by the corresponding symmetry type graph $T(\Tr(\ma))$ in Figure \ref{OPW} to define to which flag-orbit  of $\fl(\Tr(\ma))$ belong the other flags in the in the partitions $\A_2$ and $\A_1$. Finally, suppose that the face $\Psi_2$ of $\Tr(\ma)$ corresponds to a vertex in $V(\ma)$, and obtain the representative flag $\Phi_{\Psi}\in\fl(\ma)$. 
Then, determine the symmetry type of $\ma$ when $\Tr(\ma)$ is a 2-orbit or a 3-orbit map, from the types 2, $2_{01}$, $2_2$ and $3^{02}$ of maps that are vertex-transitive and also have vertices of valency 3. 

On the other hand, if $\Tr(\ma)$ is a $3k$-orbit map, applying the algorithm presented in the Figure \ref{TruncAlg} to the symmetry type graph $T(\ma)$ with $k$ vertices, we obtain straightforward the symmetry type graph of $\Tr(\ma)$ with $3k$ vertices, and we will refer to this as the \emph{truncated symmetry type graph} of $T(\ma)$. 
For instance, the truncated symmetry type graphs with 6 vertices, that correspond to the seven symmetry type graphs of maps with 2 vertices are 
depicted in Figure \ref{Trunc(2-orbM)STG}. (Recall that the notation of the symmetry types of $k$-orbit maps, with $k\leq5$, follows from \cite{MedSymTypeGph}.)
\begin{figure}[htbp]
\begin{center}
\includegraphics[width=10cm]{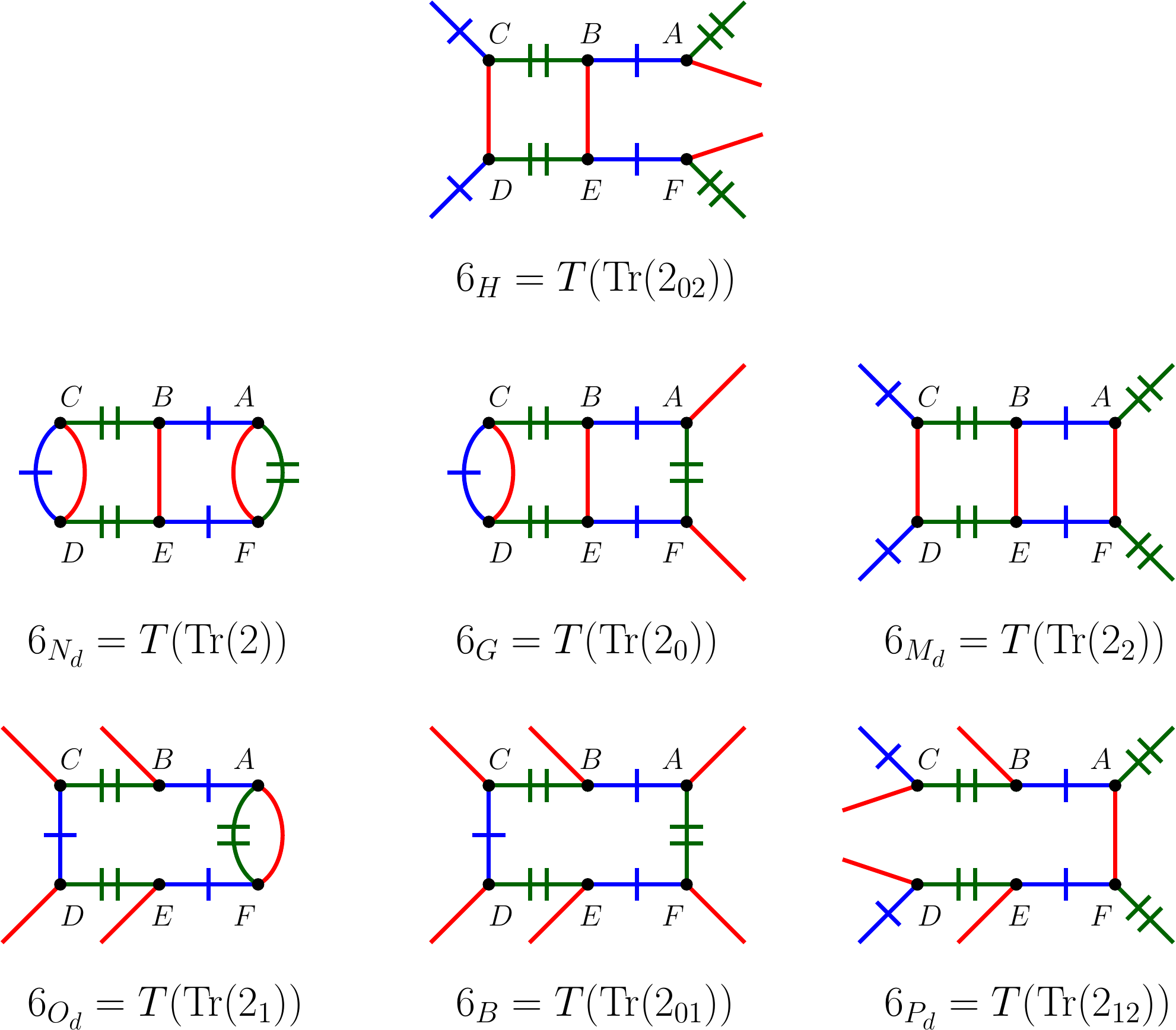}
\vspace{0cm}
\caption{Truncation symmetry type graphs with 6 vertices, of the seven 2-orbit symmetry type graphs.}
\label{Trunc(2-orbM)STG}
\end{center}
\end{figure}

Now we proceed to our first goal: to find the possible symmetry type graphs with $k$ and $\frac{3k}{2}$ vertices, of maps  that correspond to truncation of a $k$-orbit map $\ma$, with $k=4,5,6,7,9$. Extending the results of \cite{k-orbitM}, for the truncation of these $k$-orbit maps. 
To find the corresponding symmetry type of the map $\ma$, we use the same local arrangement on the flags of $\Tr(\ma)$, assembling the flags $\Psi,\Psi^2,\Psi^{2,1}\in\fl(\Tr(\ma))$ into a new flag $\Phi_{\Psi}\in\fl(\ma)$, given in \cite{k-orbitM},
under the appropriate partition $(\A_0,\A_2,\A_1)$ of the flags in $\fl(\Tr(\ma))$ according to Proposition \ref{3-0-comp}.

\subsubsection{Truncation of 4-orbit maps.}
\label{Tr4-orb}

By Theorem \ref{orbitsTr(M)}, if $\ma$ is a 4-orbit map, then its truncation $\Tr(\ma)$ is a map with either 4, 6 or 12 orbits on its flags.  

Let $\ma$ be a 4-orbit map and suppose that the map $\Tr(\ma)$ also has 4 orbits on its flags. Then, by a proper combination of the (1,2) 2-factors shown in Figure \ref{2-factors(1,2)Tr(M)}, into a symmetry type graph with 4 vertices (and looking at the twenty two symmetry type graphs with 4 vertices, \cite{MedSymTypeGph}), it can be seen that $\Tr(\ma)$
 has either symmetry type $4_{D_p}$, $4_D$ or $4_{G_d}$ (see Figure \ref{STG-4Dp_4D_4Gd}). 
\begin{figure}[htbp]
\begin{center}
\includegraphics[width=7.5cm]{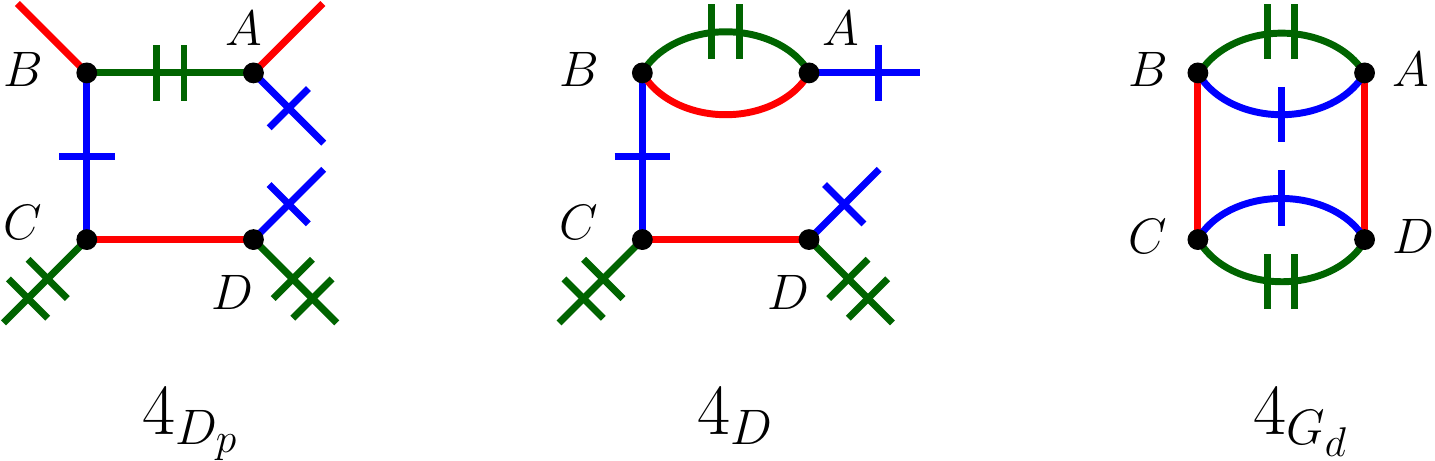}
\caption{Possible symmetry type graphs of the 4-orbit maps that correspond to the truncation of another map.}
\label{STG-4Dp_4D_4Gd}
\end{center}
\end{figure} 

Notice that the maps with the symmetry type $4_{D_p}$ have two orbits on their faces. In fact, one of its flag-orbits is completely contained in one of the face-orbits. However, this latter does not limit the choice of the face $\Psi_2\in V(\ma)$, corresponding to a flag $\Psi \in \fl(\Tr(\ma))$.
Let $(\A_0,\A_2,\A_1)$ be the partition of the set of flags $\fl(\Tr(\ma))$, and suppose that the flag $\Psi\in\A_0$ belongs to the flag-orbit $B$.  
 Then, by assembling the flags $\Psi\in\A_0$, $\Psi^2\in\A_2$, $\Psi^{2,1}\in\A_1$ of a map $\Tr(\ma)$ with symmetry type $4_{D_p}$, we obtain four different flags $\Phi_{\Psi},\Phi_{\Psi^{1}},\Phi_{\Psi^{1,0}}$ and $\Phi_{\Psi^{1,0,2,1,0,1,2}}$ in $\fl(\ma)$, depicted in Figure \ref{flagsTr(4Dp)}. Implying that $\ma$ is a 4-orbit map, with symmetry type $4_{D_p}$, as well.
\begin{figure}[htbp]
\begin{center}
\vspace{-.1cm}
\includegraphics[width=5cm]{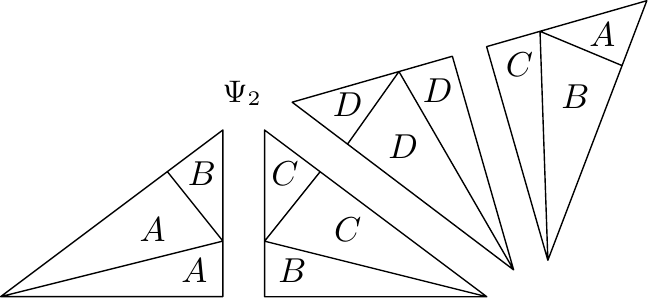}
\caption{Assembled flags in a map $\Tr(\ma)$, with symmetry type $4_{D_p}$ as truncation of another 4-orbit map. The element $\Psi_2 \in V(\ma)$ corresponds to the face of the base flag $\Psi\in\fl(\Tr(\ma))$.}
\label{flagsTr(4Dp)}
\end{center}
\end{figure} 

On the other hand, the maps with symmetry type $4_D$ and $4_{G_d}$ are face-transitive maps. Hence the arrangement of their flags is more natural than for $4_{D_p}$. In case that $\Tr(\ma)$ is of type $4_D$, we suppose that the flag $\Psi\in\A_0$ belongs to the flag-orbit $A$, consequently the flag $\Psi^2\in\A_2$ belongs to the flag-orbit $B$ and $\Psi^{2,1}\in\A_1$ belongs to the flag-orbit $C$. Inducing the local arrangement on the flags of $\Tr(\ma)$ as those in Figure \ref{flagsTr(4D-4Gd)} $(a)$. Where the other assembled flags 
\begin{itemize}
\item $\Phi_{\Psi^0}$, with the flag $\Psi^0\in\A_0$ in the orbit $B$ and the flags $\Psi^{0,2}\in\A_2$ and $\Psi^{0,2,1}\in\A_1$ in the orbit $A$;
\item  $\Phi_{\Psi^{0,1}}$, with the flags $\Psi^{0,1}\in\A_0$ and $\Psi^{0,1,2}\in\A_2$ in the orbit $C$ and the flag $\Psi^{0,1,2,1}\in\A_1$ in the orbit $B$; and
\item  $\Phi_{\Psi^{0,1,0}}$, with the flags $\Psi^{0,1,0}\in\A_0$, $\Psi^{0,1,0,2}\in\A_2$ and the flag $\Psi^{0,1,0,2,1}\in\A_1$ in the orbit $D$. 
\end{itemize}
\begin{figure}[htbp]
\begin{center}
\includegraphics[width=10cm]{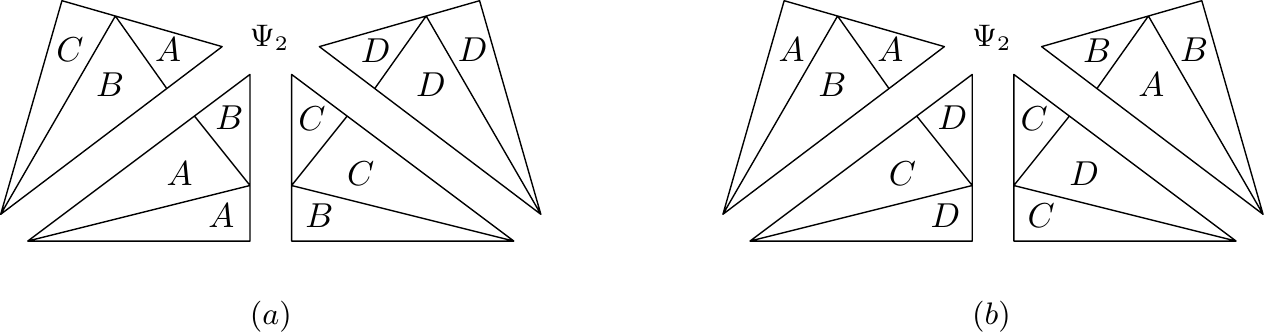}
\caption{Assembled flags in a map $\Tr(\ma)$, with $(a)$ symmetry type $4_D$, and $(b)$ with symmetry type $4_{G_d}$ as truncation of another 4-orbit map. The element $\Psi_2 \in V(\ma)$ corresponds to the face of the base flag $\Psi\in\fl(\Tr(\ma))$.}
\label{flagsTr(4D-4Gd)}
\end{center}
\end{figure} 
Proceeding in similar way with the flags of a map $\Tr(\ma)$ of type $4_{G_d}$. Assume that the flags $\Psi\in\A_0$ and $\Psi^{2,1}\in\A_1$ belong to the flag-orbit $A$, and the flag $\Psi^2\in\A_2$ belongs to the flag-orbit $B$. Then the four flags produced by this local arrangement on the flags of $\Tr(\ma)$ as those in Figure \ref{flagsTr(4D-4Gd)} $(b)$. Where the assembled flags are $\Phi_{\Psi}$;
\begin{itemize}
\item $\Phi_{\Psi^0}$, with the flags $\Psi^0\in\A_0$ and $\Psi^{0,2,1}\in\A_1$ are in the orbit $D$ and the flag $\Psi^{0,2}\in\A_2$ in the orbit $C$;
\item  $\Phi_{\Psi^{0,1}}$, with the flags $\Psi^{0,1}\in\A_0$ and $\Psi^{0,1,2,1}\in\A_1$ in the orbit $C$ and the flag $\Psi^{0,1,2}\in\A_2$ in the orbit $D$; and
\item  $\Phi_{\Psi^{0,1,0}}$, with the flags $\Psi^{0,1,0}\in\A_0$ and $\Psi^{0,1,0,2,1}\in\A_1$ in the orbit $B$ and the flag $\Psi^{0,1,0,2}\in\A_2$ in the orbit $A$. 
\end{itemize}
Therefore, if we assume that $\Tr(\ma)$ and $\ma$ are 4-orbit maps, then we obtain the following proposition.

\begin{prop}
\label{4-orbTr}
If the truncation $\Tr(\ma)$ of a 4-orbit map $\ma$ is a 4-orbit map, then exactly one of the following holds.
\begin{itemize}
\item[(i)] $\ma$ and $\Tr(\ma)$ are of type $4_{D_p}$,
\item[(ii)] $\ma$ is of type $4_E$ and $\Tr(\ma)$ is of type $4_D$, or
\item[(iii)] $\ma$ is of type $4_G$ and $\Tr(\ma)$ is of type $4_{G_d}$.
\end{itemize}
\end{prop}

Recall that given a 4-orbit map $\ma$, its truncated map $\Tr(\ma)$ might be a 6-orbit map. 
To find the possible symmetry types of maps with six orbits on its flags that would correspond to the symmetry type of a map $\Tr(\ma)$, when $\ma$ is a 4-orbit map, it is necessary to observe the following.

The symmetry type graph of a 6-orbit map $\Tr(\ma)$, constructed by a proper combination of the (1,2) 2-factors shown in Figure \ref{2-factors(1,2)Tr(M)}, contains either exactly one copy of the 2-factor (4), two copies of the 2-factor (3), or exactly one copy of the 2-factors (1), (2) and (3); joined consistently with the $(0,2)$ 2-factors in Figure \ref{2-factors(0,2)}. 
 In fact, there are sixteen symmetry type graphs with six vertices that satisfy any of the three later cases. 
Recall that in the local arrangement on the flags $\Psi, \Psi^2, \Psi^{2,1}\in \fl(\Tr(\ma))$ representing elements of the tripartition $(\A_0,\A_2,\A_1)$, we consider that the face $\Psi_2$ is an element of $V(\ma)$.

Thus, by looking at the symmetry type graphs with 6 vertices in Figures \ref{Trunc(2-orbM)STG}, \ref{STG6-orbV-trans} and \ref{STG6-orbNV-trans}, we deduce the following observations on the sixteen symmetry types, that follow the latter paragraph.

\begin{itemize}
\item[i)] $6_B$, has three face-orbits, each containing 2 flag-orbits (this corresponds to the truncation symmetry type graph of $2_{01}$);
\item[ii)] $6_{P_d}$, has three face-orbits, one with 4 flag-orbits and the others with one flag-orbit each;
\item[iii)] $6_G$, $6_H$, $6_{M_d}$, $6_{N_d}$, $6_{O_d}$, have two face-orbits, one with 4 flag-orbits and the other with 2 flag-orbits;  
\item[iv)] $6_{H_p}$, has two face-orbits, both with 3 flag-orbits;
\item[v)] $6_{J_d}$,  has two face-orbits, one with 5 flag-orbits and the other with one flag-orbit;  and
\item[vi)] $6_{B_p}$ $6_{F_d}$, $6_{G_p}$, $6_{M_{dp}}$, $6_{N_{dp}}$, $6_{O_{dp}}$, $6_{P_{dp}}$, that are transitive on their faces.
\end{itemize}

Therefore, the possible cases to be consider in manner to find whenever the map $\ma$ is a 4-orbit map such that the truncation map $\Tr(\ma)$ is a 6-orbit map are the cases $ii)$ and $iii)$ above. This is, suppose that the map $\Tr(\ma)$ has symmetry type either $6_G$, $6_H$, $6_{M_d}$, $6_{N_d}$, $6_{O_d}$, or $6_{P_d}$ (Figure \ref{Trunc(2-orbM)STG}). First consider the face-orbit in $\Tr(\ma)$ containing four flag-orbits, then
\begin{itemize}
\item[(a)] If $\Tr(\ma)$ is a 6-orbit map with symmetry type graph $6_G$, $6_{N_d}$ or $6_{O_d}$, and a partition $(\A_0,\A_2,\A_1)$ of its set of flags $\fl(\Tr(\ma))$. Let $\Psi \in \A_0$ be a flag in the orbit $A$ of $\fl(\Tr(\ma))$, then $\Psi^2\in\A_2$, and $\Psi^{2,1}\in\A_1$ are flags in the orbits $F$ and $E$, respectively in $\Tr(\ma)$. Hence, the induced flag $\Phi_{\Psi}\in\fl(\ma)$ corresponds to the first flag shown in Figure \ref{flagsTr(4)_6-orb} $(a)$.  
\item[(b)] Similarly, if $\Tr(\ma)$ is a 6-orbit map with symmetry type graph $6_H$, $6_{M_d}$ or $6_{P_d}$, and a partition $(\A_0,\A_2,\A_1)$ of its set of flags $\fl(\Tr(\ma))$. Let $\Psi \in \A_0$ be a flag in the orbit $A$ of $\fl(\Tr(\ma))$, then $\Psi^2\in\A_2$, and $\Psi^{2,1}\in\A_1$ are flags in the orbits $A$ and $B$, respectively in $\Tr(\ma)$. Hence, the induced flag $\Phi_{\Psi}\in\fl(\ma)$ corresponds to the first flag shown in Figure \ref{flagsTr(4)_6-orb} $(b)$.
\end{itemize}
\begin{figure}[htbp]
\begin{center}
\includegraphics[width=8.5cm]{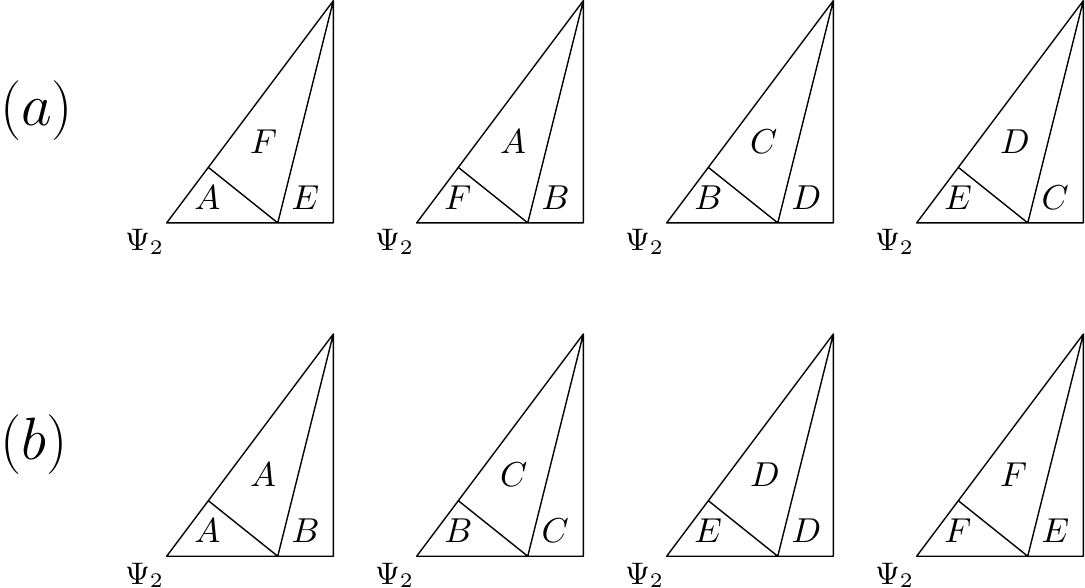}
\caption{Assembled flags in a 6-orbit map $\Tr(\ma)$, $(a)$ with symmetry type graph $6_G$, $6_{N_d}$ or $6_{O_d}$, and $(b)$ with symmetry type graph $6_H$, $6_{M_d}$ or $6_{P_d}$,as truncation of a 4-orbit map. The element $\Psi_2 \in V(\ma)$ corresponds to face of the base flag $\Psi\in\fl(\Tr(\ma))$.}
\label{flagsTr(4)_6-orb}
\end{center}
\end{figure} 
The remaining three different assembled flags in both cases: $(a)$ and $(b)$, that represent other orbits of the elements in $\fl(\ma)$, different from $\Phi_{\Psi}$, are obtained straight-forward form the adjacencies given by the symmetry type graph of $\Tr(\ma)$ and the corresponding partition $(\A_0,\A_2,\A_1)$.

Consequently, we obtain the Proposition \ref{6-orbTr_4-orbM}, and the remaining symmetry type graphs will be analysed in Section \ref{Tr6-orb}.

\begin{prop}
\label{6-orbTr_4-orbM}
If the truncation $\Tr(\ma)$ of a 4-orbit map $\ma$ is a 6-orbit map, then exactly one of the following holds.
\begin{itemize}
\item[(i)] $\ma$ is of type $4_B$ and $\Tr(\ma)$ are of type $6_{P_d}$,
\item[(ii)] $\ma$ is of type $4_C$ and $\Tr(\ma)$ is of type $6_{O_d}$, 
\item[(iii)] $\ma$ is of type $4_G$ and $\Tr(\ma)$ is either of type $6_{N_d}$ or of type $6_{M_d}$, or
\item[(iv)] $\ma$ is of type $4_H$ and $\Tr(\ma)$ is either of type $6_G$ or of type $6_H$.
\end{itemize}
\end{prop}

To conclude with the truncation of 4-orbit maps. There are twenty two truncated symmetry type graphs with 12 vertices that correspond to the twenty two symmetry type graphs with four vertices, determined by applying the algorithm in Figure \ref{TruncAlg} to each.

\subsubsection{Truncation of 5-orbit maps.}

Similarly to the previous section, if $\ma$ is a 5-orbit map, its truncation map $\Tr(\ma)$ is either a 5-orbit or a 15-orbit map, Theorem \ref{orbitsTr(M)}.

Notice that out of the thirteen different symmetry type graphs of 5-orbit maps, \cite{MedSymTypeGph}, the only possible combination of the (1,2) 2-factors  with 2 and 3 vertices in Figure \ref{2-factors(1,2)Tr(M)} is the symmetry type $5_{B_d}$ (Figure \ref{STG_5Bd}). 
This is, a 5-orbit map $\Tr(\ma)$ is the truncation of a 5-orbit map $\ma$, if $\Tr(\ma)$ is of type $5_{B_d}$. 
\begin{figure}[htbp]
\begin{center}
\includegraphics[width=3cm]{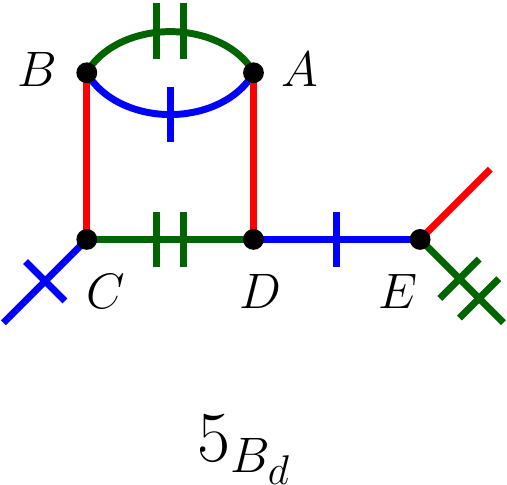}
\caption{Only possible symmetry type graph of a 5-orbit maps $\Tr(\ma)$, derived from another 5-orbit map.}
\label{STG_5Bd}
\end{center}
\end{figure} 

Let $(\A_0,\A_2,\A_1)$ be the partition on the flags of a map $\Tr(\ma)$ with symmetry type graph $5_{B_d}$, such that the flags $\Psi\in\A_0$, $\Psi^2\in\A_2$ and $\Psi^{2,1}\in\A_1$ induce the flag $\Phi_{\Psi}\in\fl(\ma)$.
 Clearly, the maps of type $5_{B_d}$ are face-transitive. Then, without lost of generality we can assume that the flag $\Psi\in\A_0$ is in the flag-orbit $A$ of $\Tr(\ma)$, where the face $\Psi_2$ is an element in $V(\ma)$. Consequently, the flags $\Psi^2\in\A_2$ and $\Psi^{2,1}\in\A_1$ are flags in the orbits $B$ and $A$, respectively. Therefore, $\Phi_{\Psi}\in\fl(\ma)$ represents an orbit of the flags of $\ma$, and the flags $\Phi_{\Psi^{1}},\Phi_{\Psi^{1,0}},\Phi_{\Psi^{0}},\Phi_{\Psi^{0,1}}\in\fl(\ma)$, with $\Psi^{1}$ a flag in the orbit $B$, $\Psi^{1,0}$ a flag in the orbit $C$, $\Psi^{0}$ a flag in the orbit $D$ and $\Psi^{0,1}$ a flag in the orbit $E$ of the elements in $\fl(\Tr(\ma))$, represent four other different flag-orbits of the set $\fl(\ma)$ as those in Figure \ref{flagsTr5-orb}.
\begin{figure}[htbp]
\begin{center}
\includegraphics[width=5cm]{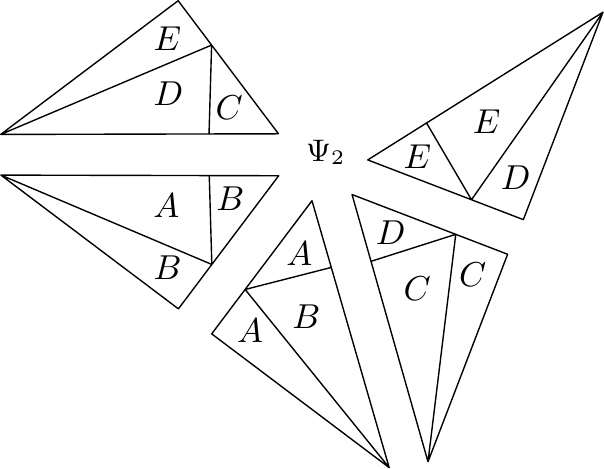}
\caption{Assembled flags of a 5-orbit map $\Tr(\ma)$ as truncation of another 5-orbit map, where $\Psi_2 \in V(\ma)$ is the corresponding face in $\Tr(\ma)$ in a base flag $\Psi\in\fl(\Tr(\ma))$.}
\label{flagsTr5-orb}
\end{center}
\end{figure} 
Obtaining the following Proposition.

\begin{prop}
If the truncation $\Tr(\ma)$ of a 5-orbit map is a 5-orbit map, then $\ma$ is of type $5_{B_p}$ and $\Tr(\ma)$ is of type $5_{B_d}$.
\end{prop}

On the other hand, there are thirteen truncated symmetry type graphs with 15 vertices that correspond to the all symmetry type graphs with five vertices shown in \cite{MedSymTypeGph}, determined by applying the algorithm in Figure \ref{TruncAlg} to each.

\subsubsection{Truncation of 6-orbit maps.}
\label{Tr6-orb}

Given a 6-orbit map $\ma$, by Theorem \ref{orbitsTr(M)} we know that its truncation map $\Tr(\ma)$ has either 6, 9 or 18 orbits on its flags. In this section we go through the remaining nine possible symmetry type of 6-orbit maps:  $6_{H_p}$, $6_{J_d}$, $6_{B_p}$, $6_{F_d}$, $6_{G_p}$, $6_{M_{dp}}$, $6_{N_{dp}}$, $6_{O_{dp}}$ and $6_{P_{dp}}$ (mentioned in Section \ref{Tr4-orb}); that might correspond to the truncation map $\Tr(\ma)$ of a map $\ma$. Further on, we find the possible symmetry type graphs with 9 vertices that correspond to the truncation $\Tr(\ma)$ of a map $\ma$ and determine which of them are related to the truncation of a 6-orbit or 9-orbit map.

Let $\Tr(\ma)$ be a map with symmetry type graph either $6_{B_p}$, $6_{G_p}$, $6_{N_{dp}}$, $6_{O_{dp}}$ (Figure \ref{STG6-orbV-trans}), $6_{F_d}$, $6_{H_p}$, $6_{J_d}$, $6_{M_{dp}}$, or $6_{P_{dp}}$ (Figure \ref{STG6-orbNV-trans}). Consider, the partition $(\A_0,\A_2,\A_1)$ on the flags of a map $\Tr(\ma)$, such that the flags $\Psi\in\A_0$, $\Psi^2\in\A_2$ and $\Psi^{2,1}\in\A_1$ induce the flag $\Phi_{\Psi}\in\fl(\ma)$.

Observe that a map $\Tr(\ma)$ with symmetry type graph $6_{B_p}$, $6_{G_p}$, $6_{N_{dp}}$, or $6_{O_{dp}}$ is transitive on its sets of faces and vertices, under the action of its automorphism group, Figure \ref{STG6-orbV-trans}.
\begin{figure}[htbp]
\begin{center}
\includegraphics[width=7cm]{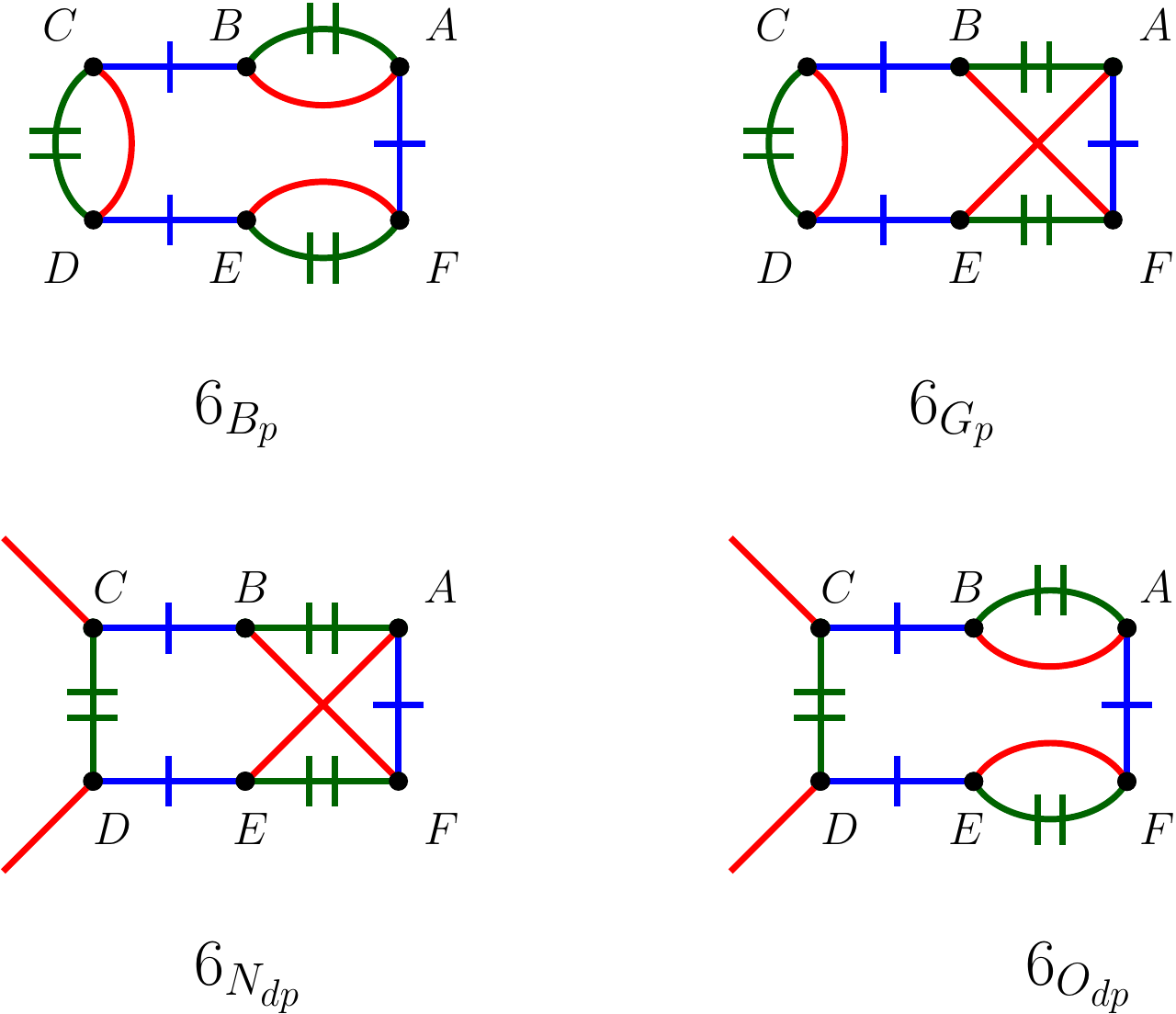}
\caption{Possible symmetry type graphs of 6-orbit vertex-transitive maps $\Tr(\ma)$, where $\ma$ is a 6-orbit map.}
\label{STG6-orbV-trans}
\end{center}
\end{figure}  
Then, we can choose any flag $\Psi$ in the flag-orbit $A$ of $\fl(\Tr(\ma))$, such that the face $\Psi_2$ is an element  in $V(\ma)$. Thus, the flags $\Psi^2\in\A_2$ and $\Psi^{2,1}\in\A_1$ are in the flag-orbits $B$ and $C$ of $\fl(\Tr(\ma))$, respectively. Hence, the flag $\Phi_{\Psi}\in\fl(\ma)$ is assembled as the very left flag in Figure \ref{flagsTr6-orb_V-trans}.
\begin{figure}[htbp]
\begin{center}
\includegraphics[width=10.5cm]{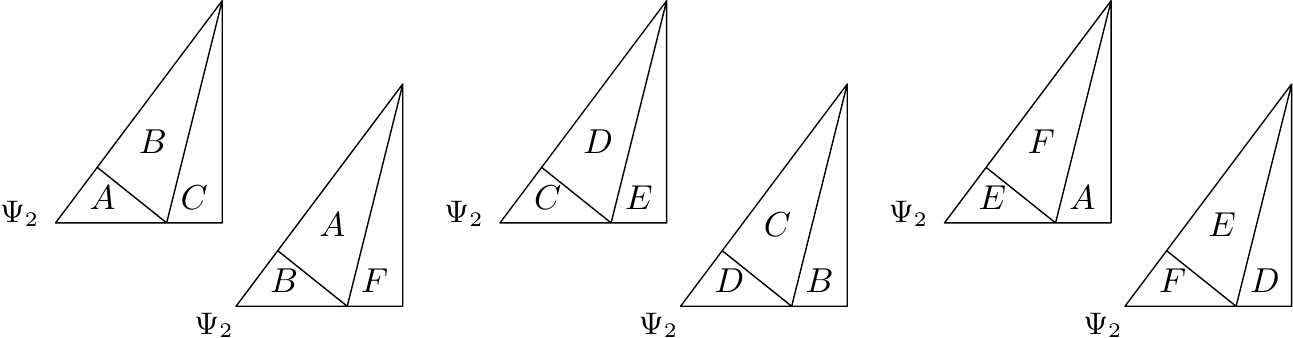}
\caption{Assembled flags of a vertex- and face-transitive 6-orbit map $\Tr(\ma)$ as truncation of a 6-orbit map, where $\Psi_2 \in V(\ma)$ is the corresponding face in $\Tr(\ma)$ in a base flag $\Psi\in\fl(\Tr(\ma))$.}
\label{flagsTr6-orb_V-trans}
\end{center}
\end{figure} 
The remaining five orbits are determined by the partition $(\A_0,\A_2,\A_1)$ of the flags in $\fl(\Tr(\ma))$ and the adjacency of the flag orbits given by each symmetry type graph in Figure \ref{STG6-orbV-trans}.

Furthermore, note that the maps with symmetry graph $6_{F_d}$ have three orbits on its vertices, while those with symmetry types $6_{H_p}$, $6_{J_d}$, $6_{M_{dp}}$ and $6_{P_{dp}}$, have two orbits on their vertices, see Figure \ref{STG6-orbNV-trans}.
\begin{figure}[htbp]
\begin{center}
\includegraphics[width=7cm]{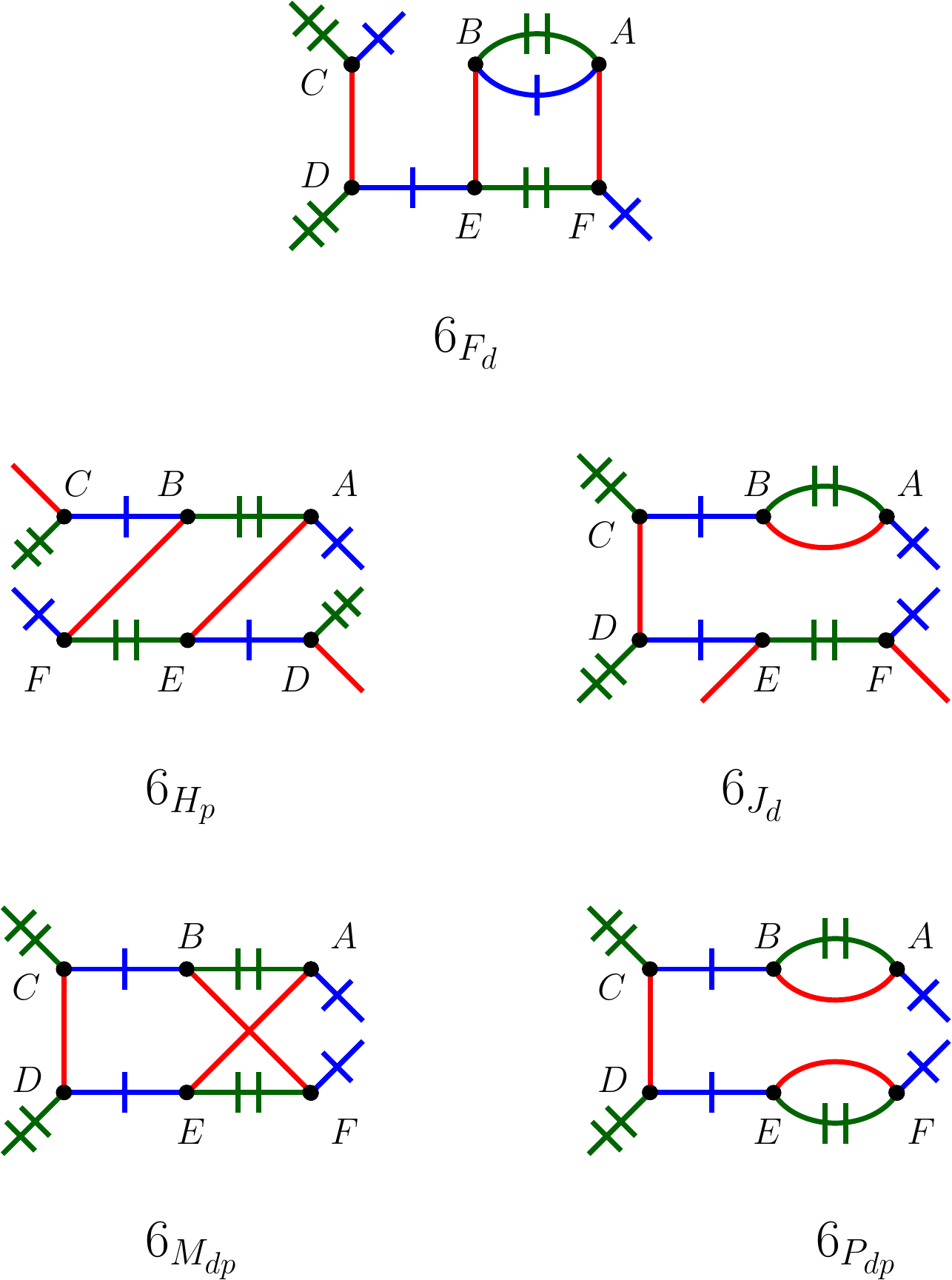}
\caption{Possible symmetry type graphs of 6-orbit maps $\Tr(\ma)$ which are not vertex-transitive, derived from 6-orbit maps $\ma$.}
\label{STG6-orbNV-trans}
\end{center}
\end{figure} 
Then, given the partition $(\A_0,\A_2,\A_1)$ of the flags of $\Tr(\ma)$. Suppose that the flag $\Psi\in\A_0$  is a flag in the flag-orbit $F$ in $\Tr(\ma)$. Assume that the face $\Psi_2$ is an element in $V(\ma)$ and $\Tr(\ma)$ has symmetry type graph either
\begin{itemize}
\item[(a)] $6_{F_d}$. Then, the flags $\Psi^2\in\A_2$ and $\Psi^{2,1}\in\A_1$ are in the flag-orbits $E$ and $B$ of $\fl(\Tr(\ma))$, respectively. Inducing a flag $\Phi_{\Psi}\in\fl(\ma)$ (as the very right one of the Figure \ref{flagsTr6-orb_NV-trans} $(a)$), that represents a flag-orbit of $\ma$. Furthermore, is not hard to see that the assembled flags $\Phi_{\Psi^{0}},\Phi_{\Psi^{0,1}},\Phi_{\Psi^{0,1,0}},\Phi_{\Psi^{0,1,0,1}},\Phi_{\Psi^{0,1,0,1,0}}\in\fl(\ma)$, with $\Psi^{0},\Psi^{0,1},\Psi^{0,1,0},\Psi^{0,1,0,1},\Psi^{0,1,0,1,0}\in\A_0$, represent the other five flag-orbits of $\ma$ following the adjacency on the flags of $\Tr(\ma)$, given by its symmetry type graph; or
\item[(b)] $6_{H_p}$, $6_{J_d}$, $6_{M_{dp}}$ or $6_{P_{dp}}$. Then, the flags $\Psi^2\in\A_2$ and $\Psi^{2,1}\in\A_1$ are in the flag-orbits $E$ and $D$ of $\fl(\Tr(\ma))$, respectively. Inducing a flag $\Phi_{\Psi}\in\fl(\ma)$ (as the very right one of the Figure \ref{flagsTr6-orb_NV-trans} $(b)$), that represents a flag-orbit of $\ma$. Respecting the partition $(\A_0,\A_2,\A_1)$, and following the adjacency on the flags in $\fl(\Tr(\ma))$ given by the respective symmetry type graph of $\Tr(\ma)$, it is not hard to find that the remaining five different assembled flags in the Figure \ref{flagsTr6-orb_NV-trans} $(b)$ represent other flag-orbits of $\fl(\ma)$, different of $\Phi_{\Psi}$.
\end{itemize}

\begin{figure}[htbp]
\begin{center}
\includegraphics[width=10.5cm]{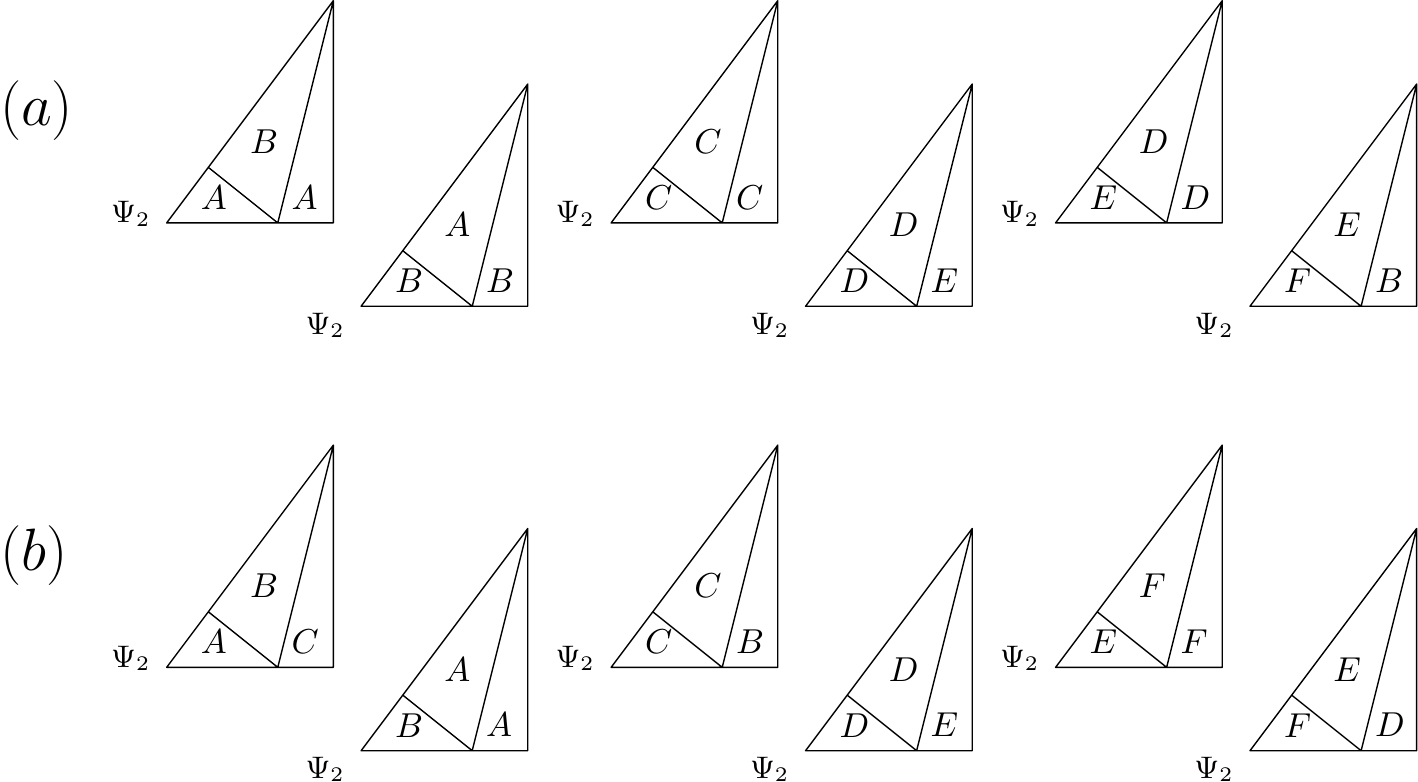}
\caption{Assembled flags of a no vertex-transitive 6-orbit map $\Tr(\ma)$ with symmetry type $6_{F_d}$ in (a), $6_{H_p}$, $6_{J_d}$, $6_{M_{dp}}$ or $6_{P_{dp}}$ in (b), as truncation of another 6-orbit map, where $\Psi_2 \in V(\ma)$ is the corresponding face in $\Tr(\ma)$ in a base flag $\Psi\in\fl(\Tr(\ma))$.}
\label{flagsTr6-orb_NV-trans}
\end{center}
\end{figure} 

Concluding with Proposition \ref{6-orbTr}. 

\begin{prop}
\label{6-orbTr}
If the truncation $\Tr(\ma)$ of a 6-orbit map is a 6-orbit map, then one of the following holds.
\begin{itemize}
\item[(i)] $\ma$ and $\Tr(\ma)$ are of type $6_{B_p}$, 
\item[(ii)] $\ma$ and $\Tr(\ma)$ are of type $6_{G_p}$, 
\item[(iii)] $\ma$ and $\Tr(\ma)$ are of type $6_{H_p}$,
\item[(iv)] $\ma$ is of type $6_{J_p}$ and $\Tr(\ma)$ is of type $6_{J_d}$,
\item[(v)] $\ma$ is of type $6_{M_{opp}}$ and $\Tr(\ma)$ is of type $6_{N_{dp}}$, 
\item[(vi)] $\ma$ is of type $6_{N_{opp}}$ and $\Tr(\ma)$ is of type $6_{F_d}$ or of type $6_{M_{dp}}$,
\item[(vii)] $\ma$ is of type $6_{O_{opp}}$ and $\Tr(\ma)$ is of type $6_{P_{dp}}$, or
\item[(viii)] $\ma$ is of type $6_{P_{opp}}$ and $\Tr(\ma)$ is of type $6_{O_{dp}}$.
\end{itemize}
\end{prop}

The truncation of a 6-orbit map $\ma$ can be a 9-orbit or a 18-orbit map. We shall say that there are seventy possible symmetry type graphs with 18 vertices, obtained by applying once more the so called algorithm in Figure \ref{TruncAlg} to each symmetry type graph $T(\ma)$ with 6 vertices. (Forty-two of the symmetry type graphs with 6 vertices are depicted together in this paper and in \cite{MedSymTypeGph}, the other twenty-two can easily be founded by applying operations described in \cite{MedSymTypeGph}, as well.)

Finally, to determine which of the many symmetry types of 9-orbit maps correspond to be the symmetry type of the truncation $\Tr(\ma)$ of a 6-orbit map $\ma$, we recall the necessary 2-factors in Figures \ref{2-factors(0,2)} and \ref{2-factors(1,2)Tr(M)}, that the symmetry type graph of the truncation of a map must contain. There are exactly ten possible symmetry type graphs with 9 vertices that satisfy the conditions to correspond to the symmetry type of a 9-orbit map $\Tr(\ma)$, as the truncation on either a 3-orbit, a 6-orbit or a 9-orbit map $\ma$, see Figure \ref{STG9-orbTr}.

\begin{figure}[htbp]
\begin{center}
\includegraphics[width=11cm]{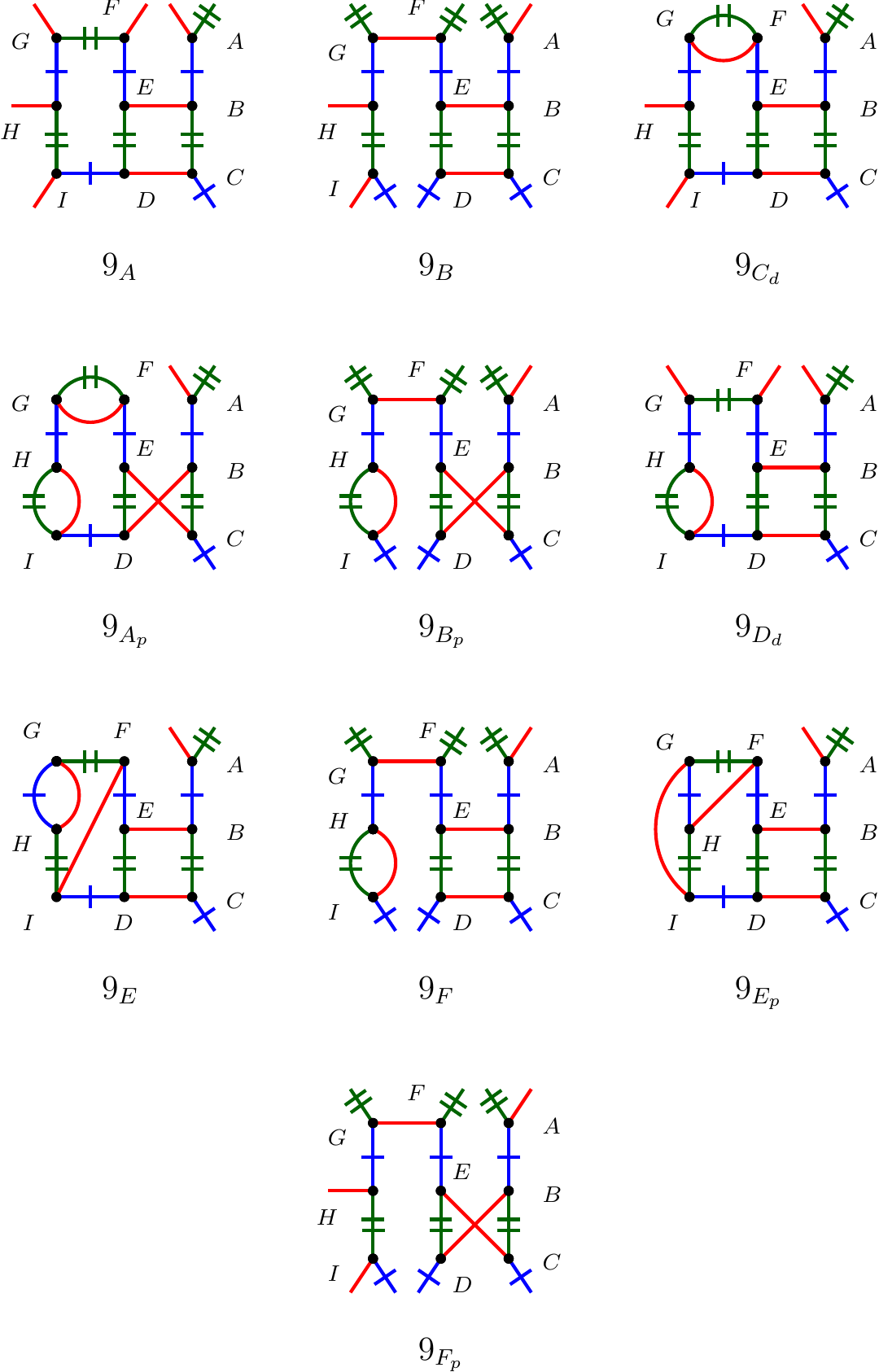}
\caption{Symmetry type graphs with 9 vertices of truncated 3-orbit, 6-orbit or 9-orbit maps.}
\label{STG9-orbTr}
\end{center}
\end{figure}

In fact, applying the algorithm in Figure \ref{TruncAlg} to the three symmetry type graphs of 3-orbit maps, it can be seen that the three graphs labelled by $9_A$, $9_B$ y $9_{C_d}$ in Figure \ref{STG9-orbTr} are isomorphic to the truncated symmetry type graphs of the symmetry type graphs $3^0$, $3^2$ and $3^{02}$, respectively. 

Moreover, if a map $\Tr(\ma)$ of type $9_A$, $9_B$ or $9_{C_d}$ has the partition $(\A_0,\A_2,\A_1)$ of the set $\fl(\Tr(\ma))$ such that a flag $\Psi\in\A_0$, is in the flag-orbit $A$ and the face $\Psi_2$ is an element of $V(\ma)$. Then, the flags $\Psi^2\in\A_2$ and $\Psi^{2,1}\in\A_1$ are in the flag-orbits $A$ and $B$. Consequently, respecting the partition $(\A_0,\A_2,\A_1)$, and following the adjacency on the flags in $\fl(\Tr(\ma))$ given by the symmetry type graphs $9_A$ or $9_{C_d}$, or $9_B$, 
we obtain six the different flags that represent the flag-orbits of $\ma$, as those shown in Figure \ref{flagsTr9-orb} $(a)$ or $(b)$, respectively.
\begin{figure}[htbp]
\begin{center}
\includegraphics[width=10.5cm]{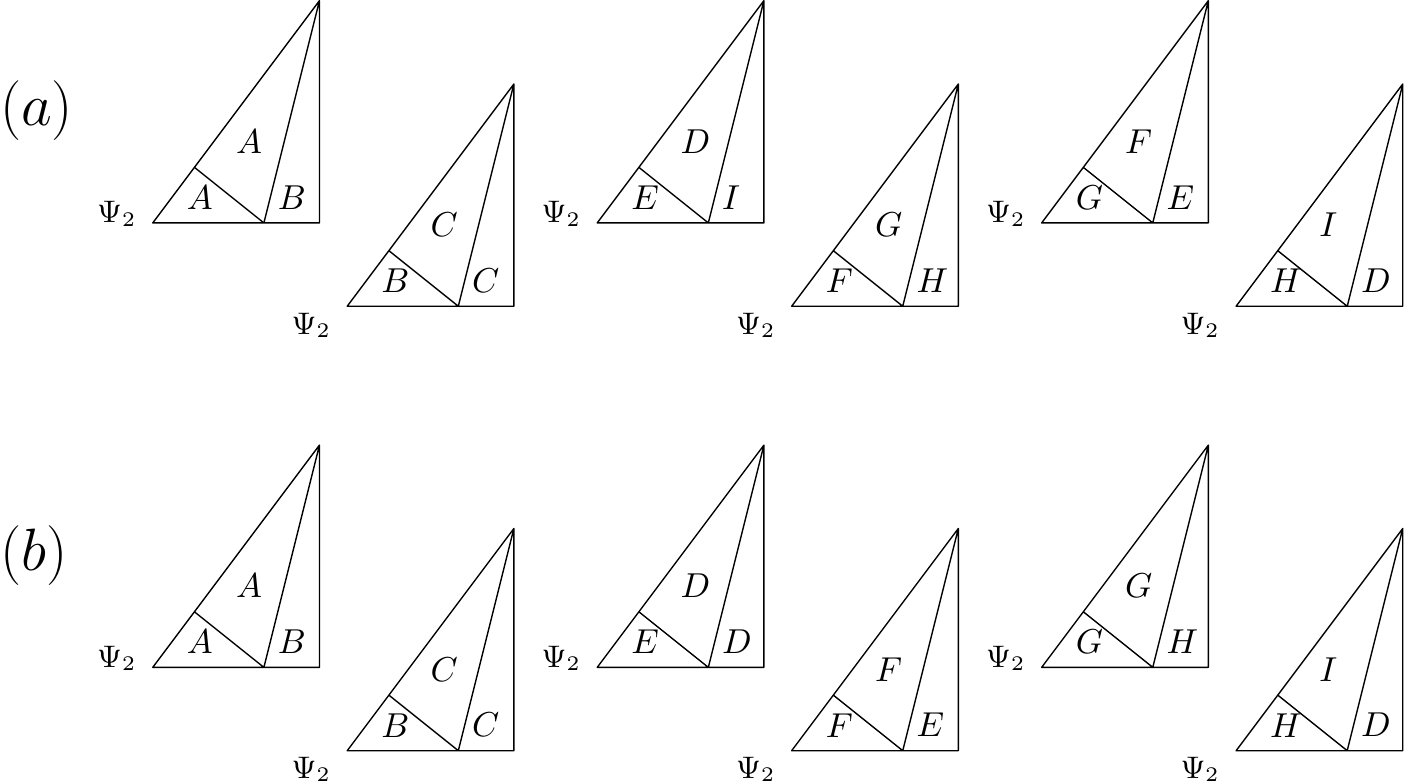}
\caption{Assembled flags of maps $\Tr(\ma)$ with $(a)$ symmetry type graph $9_A$ or $9_{C_d}$, and $(b)$ with symmetry type graph $9_B$, as truncation of a 6-orbit map, where $\Psi_2 \in V(\ma)$ is the corresponding face in $\Tr(\ma)$ in a base flag $\Psi\in\fl(\Tr(\ma))$.}
\label{flagsTr9-orb}
\end{center}
\end{figure} 
Inducing the results in Proposition \ref{9-orbTr(6)}.

\begin{prop}
\label{9-orbTr(6)}
If the truncation $\Tr(\ma)$ of a 6-orbit map is a 9-orbit map, then one of the following holds.
\begin{itemize}
\item[(i)] $\ma$ is of type $6_{D}$ and $\Tr(\ma)$ is of type $9_A$,
\item[(ii)] $\ma$ is of type $6_{F}$ and $\Tr(\ma)$ is of type $9_B$, or
\item[(iii)] $\ma$ is of type $6_{M_{opp}}$ and $\Tr(\ma)$ is of type $9_{C_d}$.
\end{itemize}
\end{prop}

It can be seen that the remaining symmetry type graphs in Figure \ref{STG9-orbTr} of 9-orbit maps, different from $9_A$, $9_B$ and $9_{C_d}$, correspond to the truncation of other 9-orbit maps, as it is shown in the following section.

\newpage
\subsubsection{Truncation of 7-orbit and 9-orbit maps.}

In this section we complete the study on truncation of $k$-orbit maps with $k \leq 7$ and $k=9$. 
In what follows we will study the symmetry types of 7-orbit and 9-orbit truncation maps $\Tr(\ma)$,
when $\ma$ is a 7-orbit or a 9-orbit map, respectively.

Once again, one can find all truncated symmetry type graphs with 21 and with 27 vertices, associated to each symmetry type graphs with 7 and 9 vertices, respectively, by applying the algorithm in Figure \ref{TruncAlg}. 

By a proper combination of the (1,2) 2-factors in Figure \ref{2-factors(1,2)Tr(M)}, it can be seen that there are exactly two different symmetry type graphs with 7 vertices: $7_J$ and $7_{J_p}$, that correspond to the symmetry type of the truncation of a 7-orbit map, depicted in the Figure \ref{STG_7J-7Jp}.
\begin{figure}[htbp]
\begin{center}
\includegraphics[width=10cm]{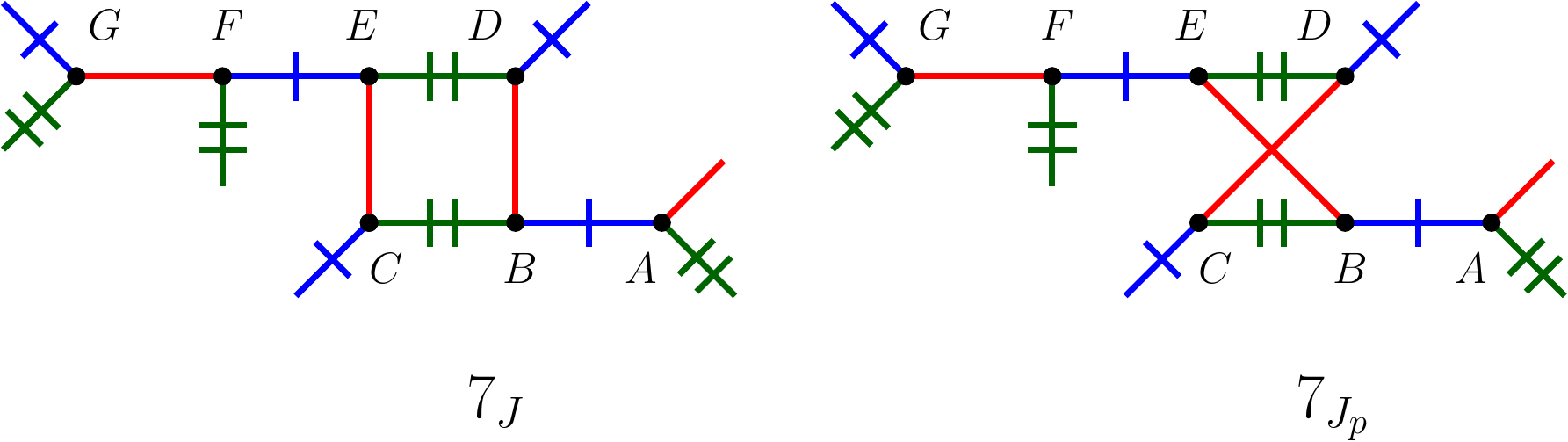}
\caption{Symmetry type graphs with 7 vertices of truncated 7-orbit maps.}
\label{STG_7J-7Jp}
\end{center}
\end{figure}
Let $\Tr(\ma)$ be a map with symmetry type either $7_J$ or $7_{J_p}$, and $(\A_0,\A_2,\A_1)$ the partition on the flags of $\Tr(\ma)$, such that the the face $\Psi_2$ of the flag $\Psi\in\A_0$ is an element of $V(\ma)$. Then, if $\Psi$ is a flag in the orbit $A$, the flags $\Psi^2\in\A_2$ and $\Psi^{2,1}\in\A_1$ are in the orbits $A$ and $B$, respectively. Thus, the assembled flag $\Phi_{\Psi}\in\fl(\ma)$ represents the very left flag of the 7 different flags in Figure \ref{flagsTr7-orb}. 
\begin{figure}[htbp]
\begin{center}
\includegraphics[width=12.5cm]{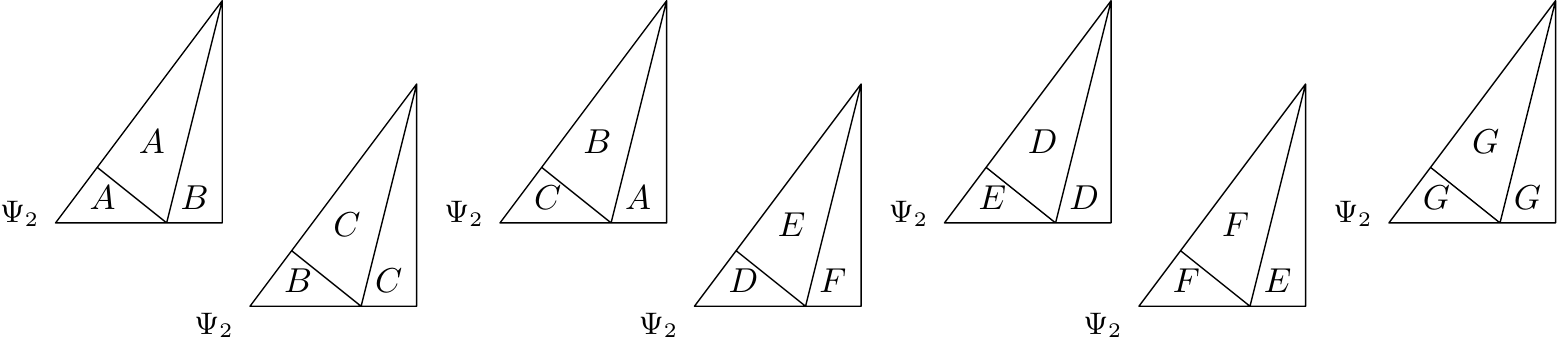}
\caption{Assembled flags of 7-orbit maps $\Tr(\ma)$ as truncation of another 7-orbit map.}
\label{flagsTr7-orb}
\end{center}
\end{figure} 
Moreover, if we respect the partition $(\A_0,\A_2,\A_1)$ of the set $\fl(\Tr(\ma))$, the adjacency of the flags given by the symmetry type graph, and the previous assumption on the flag $\Psi\in\A_0$, we can easily find the remaining six flags in Figure \ref{flagsTr7-orb} that will correspond to representing flags in $\fl(\ma)$, according to the symmetry type graph of $\Tr(\ma)$. In this way it follows Proposition \ref{7-orbTr(7)}.

\begin{prop}
\label{7-orbTr(7)}
If the truncation $\Tr(\ma)$ of a 7-orbit map is again a 7-orbit map, then one of the following holds.
\begin{itemize}
\item[(i)] $\ma$ is of type $7_{K}$ and $\Tr(\ma)$ is of type $7_{J}$, or
\item[(ii)] $\ma$ is of type $7_{L}$ and $\Tr(\ma)$ is of type $7_{J_p}$.
\end{itemize}
\end{prop}

On the other hand, recall the symmetry type graphs of 9-orbit maps, different of $9_A$, $9_B$ and $9_{C_d}$, in Figure \ref{STG9-orbTr}.
Let $\Tr(\ma)$ be a map with symmetry type $9_{A_p}$, $9_{B_p}$, $9_{D_d}$, $9_E$, $9_{E_p}$, $9_F$, or $9_{F_p}$, and consider the partition $(\A_0,\A_2,\A_1)$ of the set $\fl(\Tr(\ma))$ such that a flag $\Psi\in\A_0$ implies that $\Psi^2\in\A_2$ and $\Psi^{2,1}\in\A_1$.
Suppose that the face $\Psi_2$ is an element of $V(\ma)$, and that the flags $\Psi$ and $\Psi^2$ are flags in the flag-orbit $A$ of $\Tr(\ma)$, so as $\Psi^{2,1}$ is a flag in the flag-orbit $B$. In particular, respecting the partition $(\A_0,\A_2,\A_1)$ and following the adjacencies of the flag with respect to $\Psi$, $\Psi^2$ and $\Psi^{2,1}$ as before. If $\Tr(\ma)$ has symmetry type graph 
\begin{itemize}
\item[(a)] $9_{A_p}$, $9_{D_d}$, $9_E$, or $9_{E_p}$; or
\item[(b)] $9_{B_p}$, $9_F$, or $9_{F_p}$.
 \end{itemize}
Then we can obtain two sets of 9 different flags of the truncation map $\Tr(\ma)$ as those in Figure \ref{flags9Tr9-orb}.
\begin{figure}[htbp]
\begin{center}
\includegraphics[width=10cm]{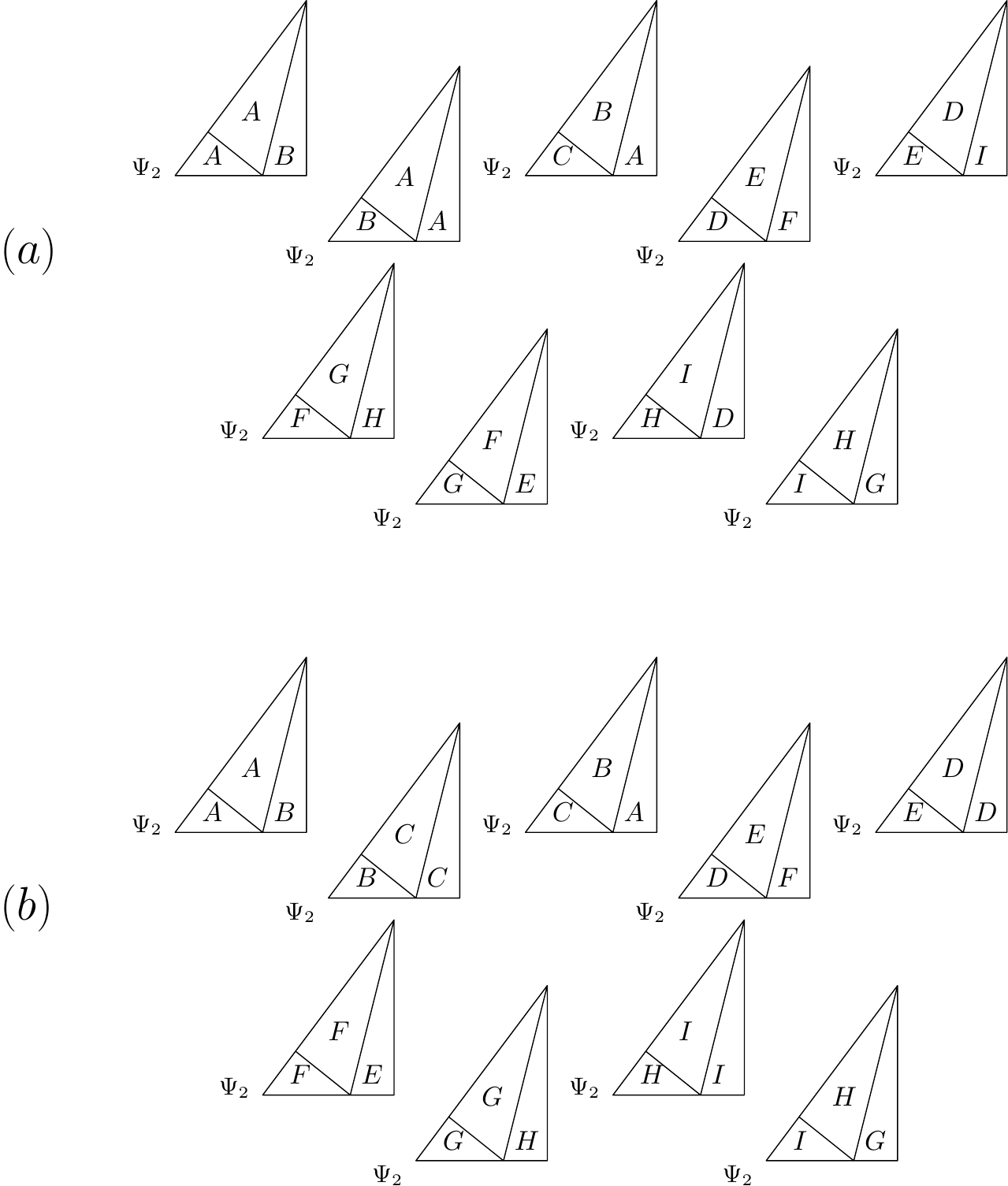}
\caption{Assembled flags of maps $\Tr(\ma)$ with $(a)$ symmetry type graph $9_{A_p}$, $9_{D_d}$, $9_E$, or $9_{E_p}$, and (b) with symmetry type graph $9_{B_p}$, $9_F$, or $9_{F_p}$, as truncation of a 9-orbit map, where $\Psi_2 \in V(\ma)$ is the corresponding face in $\Tr(\ma)$ in a base flag $\Psi\in\fl(\Tr(\ma))$.}
\label{flags9Tr9-orb}
\end{center}
\end{figure} 
Inducing the following Proposition. 

\begin{prop}
\label{9-orbTr(9)}
If the truncation $\Tr(\ma)$ of a 9-orbit map is again a 9-orbit map, then one of the following holds.
\begin{itemize}
\item[(i)] $\ma$ is of type $9_{C}$ and $\Tr(\ma)$ is of type $9_{B_p}$,
\item[(ii)] $\ma$ is of type $9_{G_p}$ and $\Tr(\ma)$ is of type $9_{E_p}$, 
\item[(iii)] $\ma$ is of type $9_{H_p}$ and $\Tr(\ma)$ is of type $9_{A_p}$.
\item[(iv)] $\ma$ is of type $9_{I}$ and $\Tr(\ma)$ is of type $9_{D_d}$,
\item[(v)] $\ma$ is of type $9_{J}$ and $\Tr(\ma)$ is of type $9_{E}$,
\item[(vi)] $\ma$ is of type $9_{K}$ and $\Tr(\ma)$ is of type $9_{F}$, or
\item[(vii)] $\ma$ is of type $9_{L}$ and $\Tr(\ma)$ is of type $9_{F_p}$.
\end{itemize}
\end{prop}

In this way, way conclude with the results obtained so far on truncation of $k$-orbit maps, with $k =1, \dots,6, 7, 9$ are listed in the Tables \ref{Tr(M)1} and \ref{Tr(M)}.

\begin{table}[h!]
\centering
\begin{tabular}{|c|c|c|c|}
\hline
Sym type  & \multicolumn{3}{|c|}{Sym type of $\Tr(\ma)$ with}     \\
 of $\ma$  &  $k$ orbits & $\frac{3k}{2}$ orbits &      $3k$-orbits    \\  
\hline
        1        &       1          &              ---                  &        $3^0$        \\ 
        2        &       2          &              ---                  &  $6_{N_d}$       \\
    $2_0$    &      ---         &              ---                  &   $6_G$             \\
    $2_2$    & $2_{12}$   &              ---                  & $6_{M_d}$       \\
    $2_1$    &      ---         &              ---                  &  $6_{O_d}$       \\ 
 $2_{01}$  &  $2_0$       &           $3^0$              &   $6_B$              \\ 
 $2_{12}$  &      ---         &              ---                  &    $6_{P_d}$     \\
 $2_{02}$  &      ---         &              ---                  &   $6_H$             \\
   $3^0$     &      ---         &              ---                  &  $9_A$             \\
   $3^2$     &      ---         &              ---                  &  $9_B$             \\
 $3^{02}$  & $3^{02}$  &              ---                  & $9_{C_d}$             \\
    $4_B$     &     ---          &        $6_{P_d}$          &      $12_B$        \\
    $4_C$     &     ---          &        $6_{O_d}$         &      $12_C$        \\
$4_{D_p}$ & $4_{D_p}$ &               ---                &      $12_D$        \\
 $4_E$        & $4_D$        &               ---                &      $12_E$        \\ 
 $4_G$        &     ---          & $6_{N_d}$, $6_{M_d}$ & $12_G$        \\
$4_{G_d}$  &   $4_G$     &                ---               &      $12_H$        \\
    $4_H$     &     ---         &  $6_G$, $6_H$            &      $12_C$        \\      
$5_{B_p}$  & $5_{B_d}$  &               ---               &      $15_A$        \\
 \hline
 \end{tabular}
\vspace{0.5cm}
\caption{Truncation symmetry types of $k$-orbit maps, with $1 \leq k \leq5$.}
\label{Tr(M)1}
\end{table}
\begin{table}[h!]
\centering
\begin{tabular}{|c|c|c|c|}
\hline
Sym type  & \multicolumn{3}{|c|}{Sym type of $\Le(\ma)$ with}     \\
 of $\ma$  &  $k$ orbits & $\frac{3k}{2}$ orbits &      $3k$-orbits    \\  
\hline
$6_{B_p}$   & $6_{B_p}$  &              ---               &      $18_B$        \\ 
 $6_D$        &    ---           &        $9_A$                 &      $18_D$        \\
 $6_F$        &    ---           &        $9_B$                 &      $18_F$        \\
$6_{G_p}$  & $6_{G_p}$  &              ---               &      $18_G$         \\ 
$6_{H_p}$  & $6_{H_p}$  &              ---               &      $18_H$         \\ 
$6_{J_p}$  & $6_{J_d}$  &              ---                &      $18_J$         \\ 
$6_{M_{opp}}$ & $6_{N_{dp}}$  & $9_{C_d}$         &       $18_M$        \\
$6_{N_{opp}}$  & $6_{F_d}$, $6_{M_{dp}}$  &    ---    &      $18_N$         \\ 
$6_{O_{opp}}$  & $6_{P_{dp}}$  &   ---             &      $18_O$         \\  
$6_{P_{opp}}$  & $6_{O_{dp}}$  &    ---            &      $18_P$         \\ 
$7_{K}$  & $7_{J}$  &              ---               &      $21_K$         \\ 
$7_{L}$  & $7_{J_p}$  &              ---               &      $21_L$         \\  
$9_{C}$  & $9_{B_p}$  &              ---                &      $27_C$         \\ 
$9_{G_{p}}$ & $9_{E_{p}}$  & ---         &       $27_G$       \\
$9_{H_{p}}$  & $9_{A_p}$  &    ---    &     $27_H$         \\ 
$9_{I}$  & $9_{D_{d}}$  &   ---             &      $27_I$         \\  
$9_{J}$  & $9_{E}$  &    ---            &      $27_J$         \\ 
$9_{K}$  & $9_{F}$  &              ---               &      $27_K$         \\ 
$9_{L}$  & $9_{F_p}$  &              ---               &      $27_L$         \\  
\hline
 \end{tabular}
\vspace{0.5cm}
\caption{Truncation symmetry types of $k$-orbit maps, with $6 \leq k \leq7$ and $k=9$.}
\label{Tr(M)}
\end{table}

\newpage
\subsection{Composition of a dual and truncation.}
\label{subsec:Lp}
For a given map $\ma$, the vertices of its truncation map $\Tr(\ma)$ have valency 3. Then, the dual map of $\Tr(\ma)$ is a map with triangular faces. Hence, there is a correspondence between the sets of vertices and faces of $\ma$ with the vertex set of the dual map $(\Tr(\ma))^*$, also known as the {\em two-dimensional subdivision} of $\ma$, \cite{PisanskiRandic}. Its symmetry type graph can easily be founded by exchanging the colours $i$ and $2-i$, with $i \in \{0,1,2\}$, on the edges of the graphs $\mathcal{G}_{\Tr(\ma)}$ and $T(\Tr(\ma))$ mentioned through the past section.

However, if we proceed in the other order and consider the truncation of the dual of a map $\ma^*$, then we produce a map isomorphic to the \emph{leapfrog map} $\Le(\ma)$ of a map $\ma$, \cite{PisanskiRandic,FowlerPisanski,Leapfrog}. In other words, $\Le(\ma) \cong \Tr(\ma^*)$. And this gives a completely different result than the map $(\Tr(\ma))^*$ as one can see below.

\subsubsection{Leapfrog map.}

As we said before, the leapfrog map $\Le(\ma)$ of $\ma$ is isomorphic to the truncation of the dual map $\ma^*$ of $\ma$. One way to construct the leapfrog map $\Le(\ma)$ of the map $\ma$ is by drawing, on the surface, a perpendicular edge to each edge of $\ma$  and joining by an edge the two end points of two edges if the corresponding edges in $\ma$ share a vertex and belong to the same face. 
In this way, we obtain a one-to-one correspondence between the faces of $\Le(\ma)$ and the set of faces and vertices of $\ma$. In Figure \ref{OctLeap} is shown the image of the octahedron after apply the leapfrog operation.
\begin{figure}[htbp]
\begin{center}
\includegraphics[width=3cm]{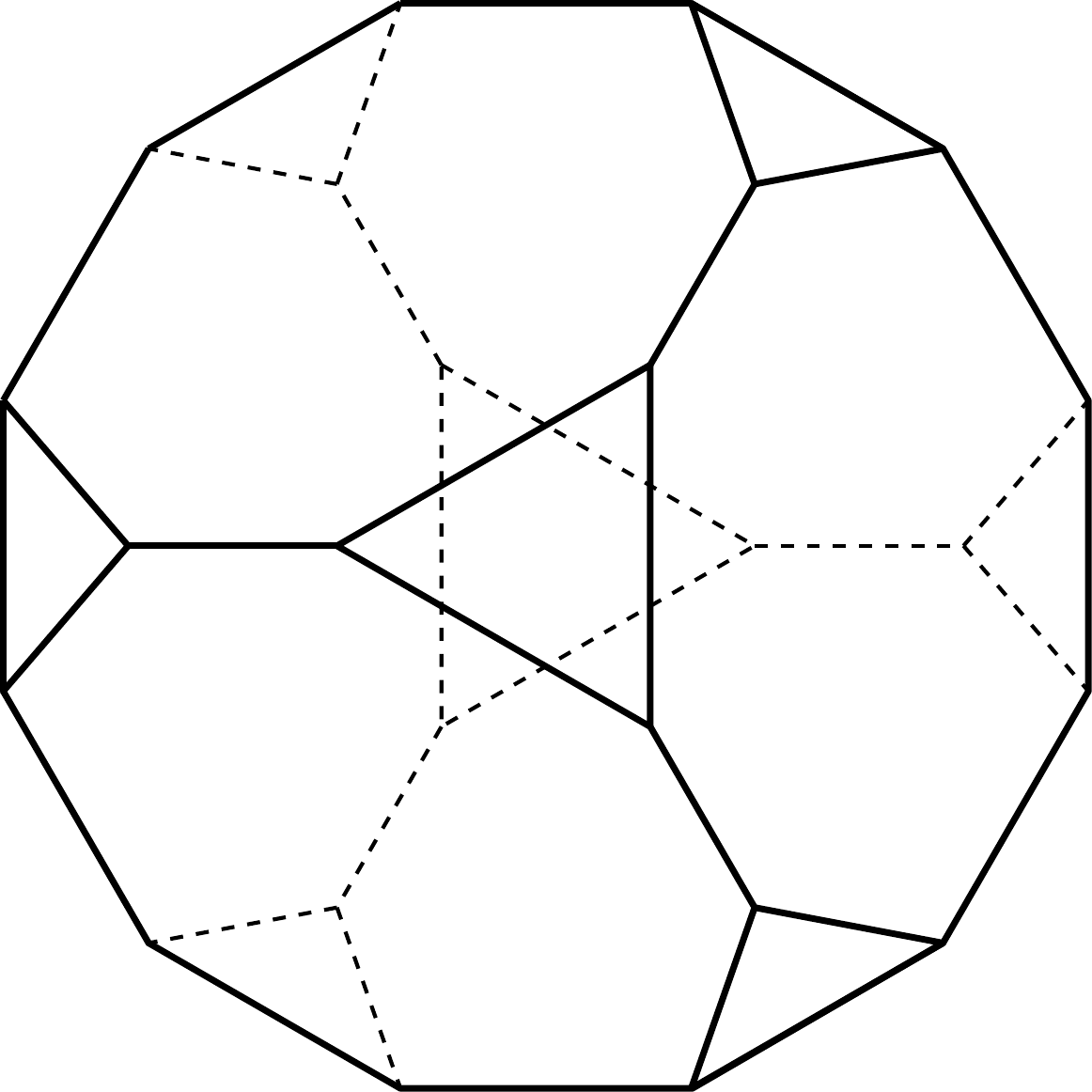}
\caption{Leapfrog image of the octahedron, isomorphic to the truncated cube.}
\label{OctLeap}
\end{center}
\end{figure}

 Notice that all the vertices of the map $\Le(\ma)$ have valence three. The faces of $\Le(\ma)$ that are in correspondence with the faces of $\ma$ remain of the same length, while the faces of $\Le(\ma)$ that correspond to the vertices of $\ma$ are of length two times the valence of its corresponding vertex. It is not hard to see that 
the map $\Le(\ma)$ contains $3|E(\ma)|$ edges. 

As is depicted in the Figure \ref{LeapfrogFlags}, every flag in $\fl(\ma)$ is divided into three different flags of the leapfrog map $\Le(\ma)$. Let $\Phi =(\Phi_0,\Phi_1,\Phi_2) \in \fl(\ma)$ be a flag in $\fl(\ma)$, 
\begin{figure}[htbp]
\begin{center}
\includegraphics[width=6cm]{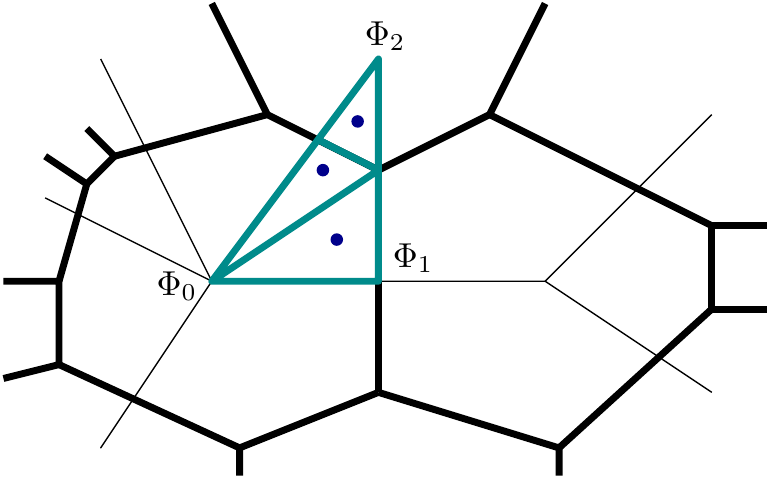}
\caption{The three respective flags of $\fl(\Le(\ma))$ to the flag $\Phi =(\Phi_0,\Phi_1,\Phi_2)  \in \fl(\ma)$.}
\label{LeapfrogFlags}
\end{center}
\end{figure}
then $(\Phi,0):=(\{\Phi_1,\Phi_2\},\Phi_1,\Phi_0\}$, $(\Phi,1):=\{\{\Phi_1,\Phi_2\},\{\Phi_0,\Phi_2\},\Phi_0\}$ and $(\Phi,2):=\{\{\Phi_1,\Phi_2\},\{\Phi_0,\Phi_2\},\Phi_2\}$ denote the three corresponding flags of $\Phi$ in $\fl(\Le(\ma))$. The adjacency between these is given as follows. 
\begin{align*}
(\Phi,0)\cdot r'_0&=(\Phi^{s_2},0), & (\Phi,0)\cdot r'_1&=(\Phi,1),             & (\Phi,0)\cdot r'_2&=(\Phi^{s_0},0); \\
(\Phi,1)\cdot r'_0&=(\Phi^{s_1},1), & (\Phi,1)\cdot r'_1&=(\Phi,0),             & (\Phi,1)\cdot r'_2&=(\Phi,2); \\
(\Phi,2)\cdot r'_0&=(\Phi^{s_1},2), & (\Phi,2)\cdot r'_1&=(\Phi^{s_0},2), & (\Phi,2)\cdot r'_2&=(\Phi,1).
\end{align*}
Note that once again $r'_0$, $r'_1$ and $r'_2$ depend of the adjacency between the flags in $\fl(\ma)$. Thus, we also present the algorithm shown in Figure \ref{LeapfrogAlg} to construct, from $\gr$, the flag graph of $\Le(\ma)$. 

\begin{figure}[htbp]
\begin{center}
\includegraphics[width=10cm]{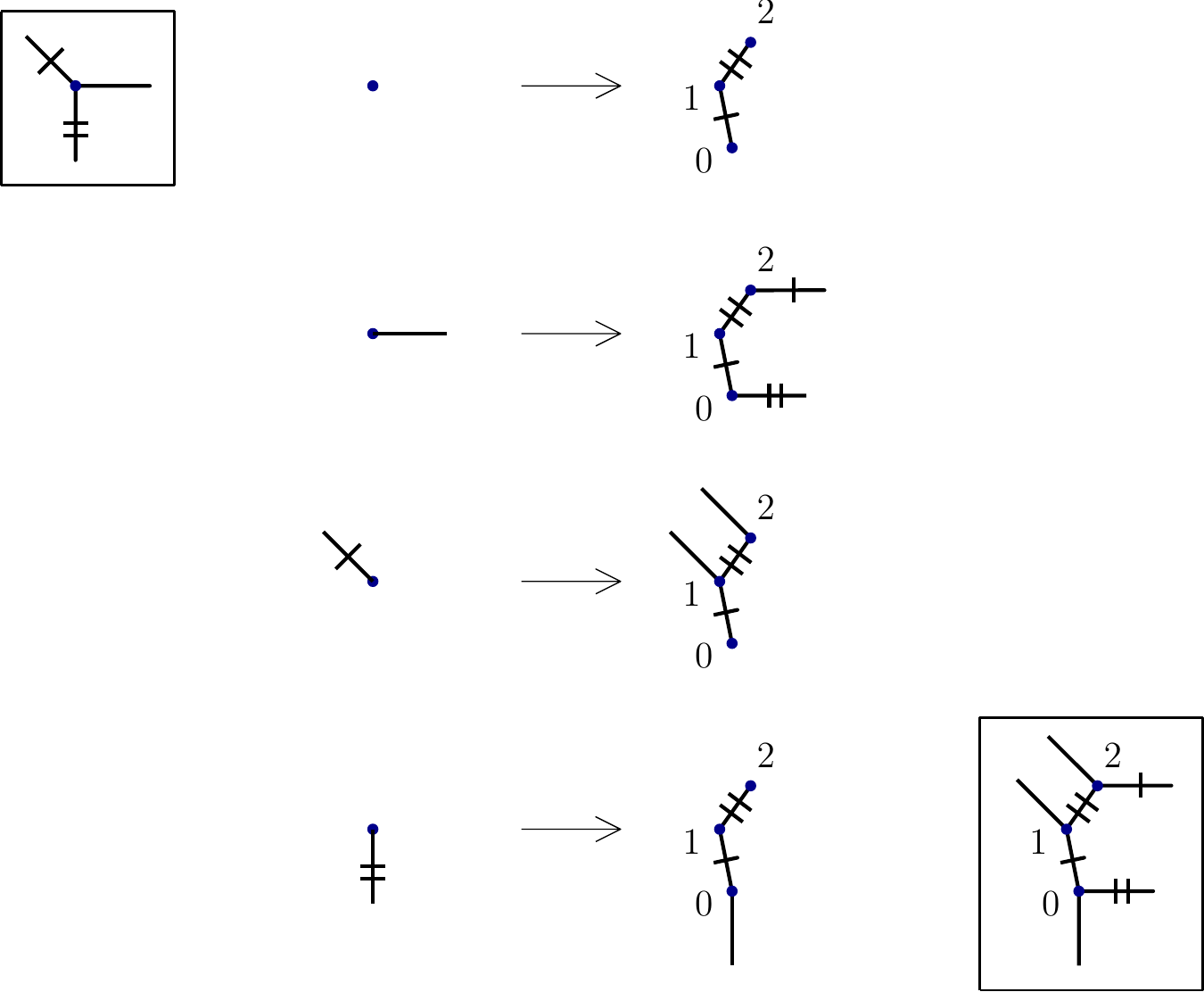}
\caption{Local representation of a flag in $\gr$, in the left. The image under the leapfrog operation, locally obtained, in the right.}
\label{LeapfrogAlg}
\end{center}
\end{figure}

Recall that the flag graph $\gr$ of a $k$-orbit map $\ma$ which truncation map $\Tr(\ma)$ is a $k$-orbit or a $\frac{3k}{2}$-orbit map, can be quotient into a graph isomorphic to $2_{01}$, \cite{k-orbitM}. Meaning that there exist a bipartition $(A,B)$ on the vertices of $\gr$ such that each vertex in the partition $A$ is adjacent to a vertex in the partition $B$ by and edge of colour 2. In addition, recall that there is a bijection $\delta: \fl(\ma) \to \fl(\ma^*)$ such that for each $\Phi \in \fl(\ma)$ and each $i \in \{0,1,2\}$, $\Phi^i\delta=(\Phi\delta)^{2-i}$. Since $\Le(\ma) \cong \Tr(\ma^*)$, we shall say that the flag graph $\gr$ of a $k$-orbit map $\ma$ such that its leapfrog map is a $k$-orbit or a $\frac{3k}{2}$-orbit map, can be quotient in to a graph isomorphic to $2_{12}$, i.e. there exist a bipartition $(A',B')$ on the vertices of $\gr$ such that each vertex in the partition $A'$ is adjacent to a vertex in the partition $B'$ by and edge of colour 0. 

However, the flag graph $'\mathcal{G}_{\Le(\ma)}$ of the leapfrog map $\Le(\ma)$, of any map $\ma$, cn be quotient into a graph isomorphic to $3^0$, with the corresponding partition $(\A_2,\A_1,\A_0)$ of the vertices of the flag graph $\mathcal{G}_{\Le(\ma)}$, in such way that for each vertex $\Psi\in\fl(\Le(\ma))$ in the partition $\A_2$, the vertices $\Psi^2$ and $\Psi^{2,1}$ correspond to the partitions $\A_1$ and $\A_2$, respectively. Thus, the assembling these flags we obtain a new flag $\Phi_{\Psi}:=\{\Psi,\Psi^2,\Psi^{2,1}\}\in\fl(\ma)$, where the face $\Psi_2$ will be considered as an element of $F(\ma)$.
 
Therefore, by the results on truncation, shown in the Tables \ref{Tr(M)1} and \ref{Tr(M)}, we can obtain the classification given in the Tables \ref{Le(M)1} and \ref{Le(M)} where are listed the symmetry types that the map $\Le(\ma)$ can have.

\begin{table}[h!]
\centering
\begin{tabular}{|c|c|c|c|}
\hline
Sym type  & \multicolumn{3}{|c|}{Sym type of $\Le(\ma)$ with}     \\
 of $\ma$  &  $k$ orbits & $\frac{3k}{2}$ orbits &      $3k$-orbits    \\  
\hline
        1        &       1          &              ---                  &        $3^0$        \\ 
        2        &       2          &              ---                  &  $6_{N_d}$       \\
    $2_0$    & $2_{12}$   &              ---                  & $6_{M_d}$       \\
    $2_2$    &      ---         &              ---                  &   $6_G$             \\
    $2_1$    &      ---         &              ---                  &  $6_{O_d}$       \\ 
 $2_{01}$  &      ---         &              ---                  &    $6_{P_d}$     \\
 $2_{12}$  &  $2_0$       &           $3^0$              &   $6_B$              \\ 
 $2_{02}$  &      ---         &              ---                  &   $6_H$             \\
   $3^0$     &      ---         &              ---                  &  $9_B$             \\
   $3^2$     &      ---         &              ---                  &  $9_A$             \\
 $3^{02}$  & $3^{02}$  &              ---                  & $9_{C_d}$             \\
$4_{B_d}$  &     ---          &        $6_{P_d}$          &      $12_B$        \\
$4_{C_d}$  &     ---          &        $6_{O_d}$         &      $12_C$        \\
$4_{D_p}$ & $4_{D_p}$ &               ---                &      $12_D$        \\
 $4_{E_d}$ & $4_D$        &               ---                &      $12_E$        \\ 
 $4_G$        &   $4_G$     &                ---               &      $12_H$        \\
$4_{G_d}$  &     ---          & $6_{N_d}$, $6_{M_d}$ & $12_G$        \\
$4_{H_d}$  &     ---         &  $6_G$, $6_H$            &      $12_C$        \\      
$5_{B_p}$  & $5_{B_d}$  &               ---               &      $15_A$        \\
 \hline
 \end{tabular}
\vspace{0.5cm}
\caption{Leapfrog symmetry types $k$-orbit maps, with $1 \leq k \leq 5$.}
\label{Le(M)1}
\end{table}
\begin{table}[h!]
\centering
\begin{tabular}{|c|c|c|c|}
\hline
Sym type  & \multicolumn{3}{|c|}{Sym type of $\Le(\ma)$ with}     \\
 of $\ma$  &  $k$ orbits & $\frac{3k}{2}$ orbits &      $3k$-orbits    \\  
\hline
$6_{B_p}$  & $6_{B_p}$  &              ---                 &      $18_B$        \\ 
$6_{D_d}$  &    ---           &        $9_A$                 &      $18_D$        \\
$6_{F_d}$  &    ---           &        $9_B$                 &      $18_F$        \\
$6_{G_p}$  & $6_{G_p}$  &              ---                 &      $18_G$         \\ 
$6_{H_p}$  & $6_{H_p}$  &              ---                &      $18_H$         \\ 
$6_{J_p}$  & $6_{J_d}$  &              ---                  &      $18_J$         \\ 
$6_{M_{dp}}$ & $6_{N_{dp}}$   & $9_{C_d}$          &       $18_M$        \\
$6_{N_{dp}}$  & $6_{F_d}$, $6_{M_{dp}}$  &    ---                &      $18_N$         \\ 
$6_{O_{dp}}$  & $6_{P_{dp}}$  &    ---                &      $18_O$         \\  
$6_{P_{dp}}$  & $6_{O_{dp}}$  &    ---                &      $18_P$         \\ 
$7_{K_d}$  & $7_{J}$  &              ---               &      $21_K$         \\ 
$7_{L_d}$  & $7_{J_p}$  &              ---               &      $21_L$         \\  
$9_{C_d}$  & $9_{B_p}$  &              ---                &      $27_C$         \\ 
$9_{G_{pd}}$ & $9_{E_{p}}$  & ---         &       $27_G$       \\
$9_{H_{pd}}$  & $9_{A_p}$  &    ---    &     $27_H$         \\ 
$9_{I_d}$  & $9_{D_{d}}$  &   ---             &      $27_I$         \\  
$9_{J_d}$  & $9_{E}$  &    ---            &      $27_J$         \\ 
$9_{K_d}$  & $9_{F}$  &              ---               &      $27_K$         \\ 
$9_{L_d}$  & $9_{F_p}$  &              ---               &      $27_L$         \\  
 \hline
 \end{tabular}
\vspace{0.5cm}
\caption{Leapfrog symmetry types $k$-orbit maps, with $6 \leq k \leq7$ and $k=9$.}
\label{Le(M)}
\end{table}

\section{Conclusion and acknowledgements}
As in \cite{MedSymTypeGph}, the symmetry type graph is used to find the possible symmetry types of the medial $k$-orbit map, with $k\leq5$. 
Here, we complete a classification of possible symmetry types for the truncation and leapfrog of $k$-orbit maps for $k\leq7$ and $k=9$, extending the results in \cite{k-orbitM}, using the symmetry type graph of a map as a tool. 
Hence, with the help of the symmetry type graph, we can continue classifying $k$-orbit maps with their image under other operations on maps.

The author would like to thank  Isabel Hubard and Toma\v{z} Pisanki for many valuable discussions, as well as their support and orientation for the completion of this work.
 I further acknowledge the support from the Slovenian Research Agency (ARRS) for the scholarship granted for my PhD.



\begin{thebibliography}{99}

\bibitem{ColorfulPolytopes}
G. Araujo, I. Hubard, D.Oliveros and E. Schulte,
{\em Colorful polytopes and graphs}
Isr. J. Math. \textbf{xxx} (2012)
 
\bibitem{CompSymTypeGraph}
 G. Brinkmann, N. Van Cleemput and T. Pisanski,
 {\em Generation of various classes of trivalent graphs}.
 Theor. Comp. Sci. (2012), in press. 

\bibitem{RegPoly}
 H. S. M. Coxeter,
{\em Regular polytopes.}
Third edition. Dover Publications, Inc., New York (1973).

\bibitem{MedSymTypeGph}
 M. Del R\'io-Francos, I. Hubard, A. Orbani\'c and T. Pisanski,
 {\em Medial Symmetry type graphs}, 
Electronic J. of Comb. {\bf 20 (3)} (2013) P29. 

\bibitem{Leapfrog}
 M. Diudea, P. John, A. Graovac, M. Primorac and T. Pisanski,
 {\em Leapfrog and Related Operations on Toroidal Fullerenes.}
 Croat. Chem. Acta, \textbf{76 (2)} (2003), 153-159.

\bibitem{DressHuson87}
 A.W.M. Dress and D. Huson,
{\em On Tilings of the Plane}.
 Geom. Dedicata, \textbf{24} (1987) 295--310.

\bibitem{DrBr96}
A. Dress, and G. Brinkmann,
{\em Phantasmagorical Fulleroids}.
MATCH Commun. Math. Comput. Chem., \textbf{33} (1996) 87--100.
 
\bibitem{rui}
R. Duarte.
 {\em 2-Restrictedly-regular hypermaps of small genus}. 
 PhD thesis, University of Aveiro, Aveiro, Portugal (2007).

\bibitem{FowlerPisanski}
 P. Fowler and T. Pisanski,
 {\em Leapfrog Transformations and Polyhedra of Clar Type.}
 J. Chem. Soc. Faraday Trans., \textbf{90 (19)} (1994), 2865-2871.

 \bibitem{2Polyh}
I. Hubard.
 {Two-orbit polyhedra from groups}.
 {\em European J. Combin.}, \textbf{31 (3)}  (2010) 943--960.

\bibitem{MonGp_Self-Inv}
 I. Hubard, A. Orbani\'c and A.I. Weiss,
 {\em Monodromy groups and self-invariance}.
 Canad. J. Math., \textbf{61 (6)} (2009) 1300--1324.

\bibitem{Lins}
 S. Lins,
 {\em Graph-Encoded Maps}.
 J. Comb. Theory, Ser. B \textbf{32 (2)} (1982) 171--181 

\bibitem{edge-trans}
J. \v{S}ir\'a\v{n}, T. W. Tucker, M. E. Watkins, 
{\em Realizing finite edge-transitive orientable maps}, 
Journal of Graph Theory {\bf 37} (2001), 1--34.

\bibitem{McMullen}
P. McMullen,
{\em Regular polytopes of nearly full rank: Addendum}
Discrete Comput. Geom. {\bf 49 (3)} (2013) 703–705. 

 \bibitem{arp} 
P. McMullen, E. Schulte,
 {\em Abstract Regular Polytopes},
 Cambridge University Press, (2002).

\bibitem{Edge-trans}
A. Orbani\'c, D. Pellicer, T. Pisansk and T. W. Tucker, 
{Edge-transitive maps of low genus}.
{\em Ars Math. Contemp.} \textbf{4 (2)} (2011) 385--402.

\bibitem{k-orbitM}
 A. Orbani\'c, D. Pellicer and A.I. Weiss, 
 {\em Map operation and $k$-orbit maps}.
 J. Combin. Theory, Ser. A \textbf{117 (4)}  (2009) 411--429.

\bibitem{Daniel-chiral} 
D. Pellicer.
{Developments and open problems on chiral polytopes.} 
{\em Ars Math. Contemp.} \textbf{5} (2012) 333--354.

\bibitem{PisanskiRandic}
 T. Pisanski and M. Randi\' c,
 {\em Bridges between geometry and graph theory.}
 Geometry at work, MAA Notes, 53, 
 Math. Assoc. America, Washington, DC, (2000) 174--194.

\bibitem{ArchimSolids}
J. Su\'arez, E. Gancedo, J.M. \'Alvarez and A. Mor\'an,
{\em Truncating and chamfering diagrams of regular polyhedra}
 J. Math. Chem., \textbf{46}, 1, 155-163 (2009).

\bibitem{CombMaps}
A. Vince,
{\em Combinatorial Maps}
 J. Combin. Theory, Ser. B \textbf{34}  (1983) 1-21.


\end{thebibliography}
\end{document}